\documentclass[12pt]{article}
\usepackage{amssymb}
\usepackage{url}
\usepackage{amsfonts}
\usepackage{amsmath}
\usepackage{graphicx}
\graphicspath{ {./} }
\usepackage{xcolor}
\newtheorem{theorem}{Theorem}

\newtheorem{corollary}[theorem]{Corollary}

\newtheorem{definition}[theorem]{Definition}

\newtheorem{lemma}[theorem]{Lemma}

\newtheorem{proposition}[theorem]{Proposition}
\newtheorem{remark}[theorem]{Remark}

\newenvironment{proof}[1][Proof]{\noindent\textbf{#1.} }{\ \rule{0.5em}{0.5em}}
\begin{document}

\title{A logarithm law for nonautonomous systems fastly converging to equilibrium and mean field coupled systems.}
\author{Stefano Galatolo \\
%EndAName
Dipartimento di Matematica \\ Centro Interdipartimentale per lo Studio dei Sistemi Complessi \\ Università di Pisa, Italy, 56127\\
stefano.galatolo@unipi.it \and Davide Faranda \\
%EndAName
Laboratoire des Sciences du Climat et de l'Environnement,\\UMR 8212 CEA-CNRS-UVSQ, Universit\'e Paris-Saclay,\\ IPSL CEA Saclay,\\ l'Orme des Merisiers, 91191, Gif-sur-Yvette, France\\
davide.faranda@lsce.ipsl.fr}

\maketitle

\begin{abstract}
We prove that if a nonautonomous system has in a certain sense a fast
convergence to equilibrium (faster than any power law behavior) then the
time $\tau _{r}(x,y)$ needed for a typical point $x$ to enter for the
first time in a ball $B(y,r)$ centered at $y$, with small radius \ $r $ scales as the local dimension of the equilibrium measure \ $\mu $ at $y$, i.e.%
\begin{equation*}
\underset{r\rightarrow 0}{\lim }\frac{\log \tau _{r}(x,y)}{-\log r}%
=d_{\mu }(y).
\end{equation*}

We then apply the general result to concrete systems of different kind,
showing such a logarithm law for asymptotically autonomous solenoidal maps
and mean field coupled expanding maps.
\end{abstract}
\renewcommand{\thefootnote}{\fnsymbol{footnote}} 
\renewcommand{\thefootnote}{\arabic{footnote}}

{\bf \noindent  Several statistical methods to assess the rarity of an event are based on the fact that the timescale in which the event is expected to occur in the evolution of a system is approximately the inverse of the probability of the event itself. One of the possible formalizations of this relation is also called "Logarithm Law".
Results of this kind have been proved for many systems, however the relation is not always true even for chaotic systems. In the context of the study of climate change it is important to understand under which assumptions the above kind of relation holds in the non autonomous case.  We address this question, showing that the relation holds for asymptotically autonomous systems with a fast convergence to equilibrium and other similar classes of systems, including some mean field coupled ones. }

\section{Introduction}

One way to express the rarity of an event in some evolving system is to estimate the time scale in which the event is likely to occur, given the current situation of the system, and thus given some information on its initial
condition.

In the context of dynamical systems this naturally leads to the study of waiting times or hitting times indicators and to the study of the
hitting time distribution, which is in turn connected to the classical theory of extreme events (see \cite{lucarini2016extremes,faranda2024Statistical} for a survey on these topics with a particular focus on dynamical systems theory).

Most of the results already established in this direction are related to autonomous dynamical systems or stationary processes. Many important natural and social phenomena are characterized by the fact that the parameters describing the dynamics of interest may evolve with time, and the associated systems are then not autonomous. This is particularly relevant in the study of climate models and in particular in connection with the study of climate change.
Due to its profound impact on society, the study of extreme events is also particularly important in the context of climate and meteorological studies.
The study of extremal events in nonautonomous systems is then highly motivated.
 In this case the theoretical study is still in its infancy, and it is not clear under which assumptions results similar to the ones currently used in the autonomous case can be established.
The kind of non-autonomous systems we study having in mind the application to climate dynamics (and to the mean field dynamics) is the so-called {\em sequential} one, where the parameters change in time in a certain deterministic way and not the {\em random} one, in which the parameters vary randomly according to some stationary law.

In this article we focus on one of the most basic results, linking the time scale in which some rare event is likely to occur with the {\em fractal} dimension of the system in a neighborhood of the event itself, expressed by the so-called {\em local dimension}. Let $X$ be the phase space in which our dynamics occur. We will always assume that $(X,d)$ is a compact metric space. Let $x_0,x_1,...\in X$ be a trajectory of our system with initial condition $x_0$, let $y\in X $ be a target point.
Let \begin{equation*}
\tau _{r}(x_0,y)=\min \{n\in \mathbb{N}:d(x_n,y)<r\}
\end{equation*}
be the time needed for the trajectory starting from $x_0$ to enter a target of radius $r$ centered in $y$.
In the context of autonomous dynamics, supposing the system generating the trajectory has an invariant measure $\mu$, in many cases of having fast speed of mixing 
the following result can be proved: for $\mu$ almost all initial conditions $x_0$
\begin{equation}\label{loglaww}
\underset{r\rightarrow 0}{\lim }\frac{\log \tau _{r}(x_0,y)}{-\log r}%
=d_{\mu }(y)
\end{equation}
where 
\begin{equation}\label{dim} d_{\mu }(y):=\underset{r\rightarrow 0}{\lim }\frac{\log \mu(B_{r}(y))}{\log r}
\end{equation}
is the local dimension of $\mu$ at $y$ and $B_{r}(y)$ denotes the ball with center $y$ and radius $r$. This kind of result was also called a logarithm law and relates the scaling behavior of the hitting time on small targets, with the one of the measure of the targets themselves, given by the local dimension.
A logarithm law is a weaker result with respect to the ones on distribution of hitting times and extreme values theory. This result is also somewhat weaker with respect to the so-called dynamical Borel-Cantelli results (see \cite{GK}). 
In the autonomous case, logarithm laws were established for the geodesic flows and similar systems (see e.g. \cite{Sul,KM,GN}), similar results have been established for Lorenz-like flows (\cite{AGP}, \cite{GP10}, \cite{GPN}) or infinite systems (\cite{GHPZ}). 
Generally speaking, these types of statements hold true  for systems having a superpolynomial decay of correlations  
 (\cite{gal07}) even for targets that are not balls  (\cite{G10}). 
Logarithm laws, however, also hold for systems which are not chaotic like
rotations and interval exchanges. In this case, their behavior is related to the arithmetical
properties of the system (\cite{KM}, \cite{kimseo}, \cite{GP??}, \cite{BC}).
Deep relations have indeed been shown with diophantine approximation (see e.g. \cite{HV}, \cite{GP??}, \cite{GSR}).
It is worth remarking that (relatively slowly) mixing systems are known for which a logarithm law does not hold at all, and the time needed for a typical orbit to hit a small target is much larger than the inverse of the measure of the target  (see  \cite{GP??}, \cite{GSR}).

In the paper \cite{freitas2017extreme} Extreme Values Theory results are established  with the  aim of  application  to non autonomous dynamical systems. In \cite{freitas2017extreme} a previous approach of  \cite{Hus} is adapted, by weakening the uniform mixing condition that was previously used to a non uniform condition which can be verified in the context of dynamical systems.  The paper \cite{freitas2017extreme} establishes Extreme Values Laws and exponential distribution of the hitting times for a class of sequential dynamical systems whose transfer operators satisfy uniformly a list of assumption which usually are used to establish the spectral gap for those operators on a Banach space of absolutely continuous measures. This leads to application to a sequential composition of (multidimensional) expanding maps. 
The result is hence particularly interesting in the context of non autonomous dynamics, but cannot be applied to the case of systems having fractal attractor, whose dimension plays an important role in the study of the event's rarity, which is the main goal of this paper.

The link between the scaling behavior of the occurrence of the hitting time and the local dimension already established in the autonomous case  was successfully used  in climate science to estimate the rarity of given events. Logarithm laws and the the results coming from extreme value theory were used as theoretical tools to interpret empirical data and validate the use of certain statistical estimators~\cite{faranda2017dynamical,faranda2023persistent, faranda2013extreme,faranda2020diagnosing,caby2022matching}.

In non-autonomous systems, where the governing equations evolve with time, this can lead to a time-varying hitting time statistics and traditional methods for analyzing hitting time distributions may not directly apply. Moreover, the presence of external forcing or environmental perturbations further complicates the analysis, potentially leading to  deviations from expected hitting time behaviors, as highlighted numerically in~\cite{faranda2019hammam}. 
In the context of climate change applications, where understanding the timing and occurrence of extreme events is crucial, these extensions are particularly pertinent. This paper contributes to this endeavor by demonstrating the existence of a logarithm law for hitting times in certain non-autonomous systems, shedding light on the dynamics of rare events in evolving environments.

In the main result of the paper  (see Theorem \ref{main}) we consider a sequential nonautonomous deterministic dynamical system $(X,T_i)$ where $i\in \mathbb{N}^*$ and $T_i:X\to X$. Supposing that he system has a fast convergence to equilibrium to some measure $\mu$ we show that typical trajectories satisfy a logarithm law in the sense of \eqref{loglaww}.

The convergence to equilibrium notion which we consider is based on the convergence of the iterates $L_{T_n}\circ ...\circ L_{T_1}(\mu_0)$ of some initial reference measures $\mu_0$ in a certain space, to some {\em equilibrium} measure $\mu$ by the sequential composition of transfer operators $L_{T_i}$ associated to the maps $T_i$.
Some important class of systems where one is led to consider a sequential composition of maps behave like this.
We then indeed apply our main general theorem to a class of {\em asymptotically autonomous} solenoidal maps which can have fractal attractors of different dimensions (See Section \ref{yp}) and to a class of mean field coupled systems having exponential convergence to equilibrium (See Section \ref{coupmap}).
We remark that asymptotically autonomous systems, in which the considered maps have a certain limit map $T_i\to T_0$, have been proposed in \cite{Ash} and \cite{Ash2} as natural kind of models to study to understand tipping points and the statistical properties of climate change.
The concept of physical invariant measures for slowly varying non autonomous systems is reviewed in \cite{Y17}. The existence of an absolutely continuous invariant measure for the limit map in asymptotically autonomous systems has been studied in \cite{GBG19}.

\bigskip

\section{A logarithm law in the nonautonomous case}\label{sec2}

Let us introduce some notation and terminology that will be used in the
following: let us consider two compact metric spaces $X,Y$. Without loss of generality we will suppose that the diameter of $X$ and $Y$ is $1$.
Let us consider the spaces of
Borel probability measures $PM(X),~PM(Y)$ on $X$ and $Y,$ and a Borel
measurable $F:X\rightarrow Y$. We denote the pushforward of $F$  as $%
L_{F}:PM(X)\rightarrow PM(Y)$, defined by the relation%
\begin{equation*}
\lbrack L_{F}(\mu )](A)=\mu (F^{-1}(A))
\end{equation*}%
for all $\mu \in PM(X)$ and  measurable set $A\subseteq Y$. With the same
definition, the pushforward can be extended as a linear function $%
L_{F}:SM(X)\rightarrow SM(Y)$ from the vector space of Borel signed measures
on $X$ to the same space on $Y$. In this case $L_{F}$ is linear and will also be called
the transfer operator associated with the function $F$.

Let us consider on $(X,d),$ a family of maps $T_{i}:X\rightarrow X$ with $%
i\in \mathbb{N}^{\ast }$ and a sequential non autonomous system $(X,T_{i}).$

Consider two points $x,y\in X.$ The orbit of $x$ is the sequence $$
x,T_{1}(x),T_{2}(T_{1}(x)),...$$ 
 We denote the  sequential composition of the maps by  $T^{(0)}(x)=x$
and inductively for $k\in \mathbb{N}^{*}$,
$
T^{(k)}(x):=T_{k}(T^{(k-1)}(x))$.
Let $B_{r}(y)$ be the ball of radius $r$ centered in $y$, we denote
the hitting time of $B_{r}(y)$ for the orbit of $x$ as%
\begin{equation*}
\tau _{r}(x,y)=\min \{n\in \mathbb{N}:T^{(n)}(x)\in B_{r}(y)\}.
\end{equation*}

Typically $\tau _{r}(x,y)\rightarrow +\infty $ \ as $r\rightarrow 0$. To
give an estimate on how rare is the hitting of such small targets as an
event on our system, in the following we will estimate the speed, asymptotically $\tau
_{r}(x,y)$ goes to $\,+\infty .$
%
%To this purpose we define the hitting time indicators as%
%\begin{equation*}
%\underline{R}(x,x_{0})=\underset{r\rightarrow 0}{\lim \inf %}\frac{\log \tau
%_{r}(x,x_{0})}{-\log r},\overline{R}%(x,x_{0})=\underset{r\rightarrow 0}{\lim
%\inf }\frac{\log \tau _{r}(x,x_{0})}{-\log r}.
%\end{equation*}%
%These indicators estimate its exponent when $\tau _{r}%(x,x_{0})$ go to $%
%\,+\infty $ as a power law. We will see that this exponent is related to the
%local dimension of an equilibroum measure for the system.

In \ a system which is not autonomous there is not an invariant measure, 
we will replace it with a kind of asymptotically invariant one, which we
will call the equilibrium measure.

We will suppose that in our phase space $X$ a starting "reference" Borel
probability measure $\mu _{0}$ is considered (it can be for example the
normalized volume measure when $X$ is a Riemannian manifold), and that the \
iterates of the pushforward of $\mu _{0}$ \ trough the dynamics converge to
a certain measure $\mu $.
We will suppose that
there is a certain $\mu \in PM(X)$ such that \ as $n\rightarrow \infty $ 
\begin{equation*}
L_{T^{(n)}}\mu _{0}\rightarrow \mu
\end{equation*}%
with convergence in a certain topology, and with a certain superpolynomial  speed.

To formalize the assumptions, let us define a certain weak norm and distance to be considered in spaces of measures on metric spaces.
Let $(X,d) $ be a compact metric space and let $g:X\longrightarrow \mathbb{R}$ be a Lipschitz function and let $Lip(g)$
be its best Lipschitz constant, i.e. 
\begin{equation*}
\displaystyle{Lip(g)=\sup_{x,y\in X}\left\{ \dfrac{|g(x)-g(y)|}{d(x,y)}%
\right\} }.
\end{equation*}
We also define the Lipschitz norm of $g$ as
\begin{equation*}
||g||_{Lip}=\max(Lip(g),\sup_{x\in X} |g(x)|).
\end{equation*}

\begin{definition}\label{wasserstein}
Given a Borel signed measure $\mu $  on $X,$ we define a 
\textbf{Wasserstein-Kantorovich Like} norm of $\mu $  by%
\begin{equation}
||\mu ||_W =\sup_{||g||_{Lip}\leq1 }\left\vert \int {g}%
d\mu  \right\vert .
\end{equation}%
To this norm one can associate the distance
\begin{equation}
W(\mu,\nu) = ||\mu-\nu||_W
\end{equation}%
for $\mu,\nu \in SM(x)$.
\end{definition}

Let us denote the sequential composition of transfer operators $L_{T_{k}}:SM(X)\rightarrow SM(X)$ associated to the maps $T_{k}$ as   $$L^{(j,k)}:=L_{T_{k}}\circ
L_{T_{k-1}}\circ ...\circ L_{T_{j}}$$ for $k>j$ and
$$L^{(k)}:=L_{T_{k}} \circ ...\circ L_{T_{1}}$$ for $k>1$.
Coherently we denote $L^{(k,k)}:=L_{T_{k}}$, $L^{(0)}:=Id$ and $L^{(1)}:=L_{T_{1}}.$

Now we can formalize the general framework in which our abstract result is stated. As usual in the study of transfer operators we consider the action of the operator itself on a suitable normed vector space of measures or distributions. We suppose  that the space considered, which we will denote by $B_s$ has a topology which is stronger than the one induced by the $W$ distance above defined.

\begin{definition}\label{def1}
Let $(B_{s},||~||_{s})\subseteq SM(X)$ be a normed
vector subspaces of the space of Borel signed measures on $X$.
Suppose there is $C\geq 0$
such that $||~||_{W}\leq C||~||_{s}$.
Suppose that
for each $i,$ $L_{T_{i}}$ preserves $B_{s}$. We say that the nonautonomous
system $(X,T_{i})$ has weak convergence to equilibrium with
superpolynomial speed if there is a probability measure $\mu \in B_{s}$ \ and $\Phi $
superpolynomially decreasing \footnote{We say $\Phi $ is superpolynomially  decreasing if the function $\Phi :\mathbb{N\rightarrow R}$ is decreasing and \ for
each $\alpha >0$, $\lim_{n\rightarrow \infty }n^{\alpha }$ $\Phi
(n)=0$. } such that $\forall k,j\in \mathbb{N} $ with $k\geq j $ and each probability measure $\mu_0\in B_s $ 
\begin{equation*}
||\mu -L^{(j,k)}\mu _{0}||_{W}\leq \Phi (k-j)max(1,||\mu -\mu _{0}||_{s}).
\end{equation*}
\end{definition}

We stress that in the above definition we test the convergence to equilibrium  where $k, j$ varies and the speed depends on $k-j$. In this way one can have a bound on the convergence when iterating the operators at different starting times.

\begin{definition}\label{def2}
We say that a set $A\subseteq $ $B_{s}$ has uniformly bounded Lipschitz
multipliers if there is \ $C_{A}\geq 0$ depending on $A$ such that for each $%
\mu _{0}\in A$ and $\phi \in Lip(X)$ \ we have $\phi \mu _{0}\in B_{s}$ and%
\begin{equation*}
||\phi \mu _{0}||_{s}\leq C_{A}||\phi ||_{Lip}.
\end{equation*}
\end{definition}

To better explain this definition, we remark as an example that in Section \ref{coupmap}, considering expanding maps, we will choose as strong space $B_s$ a space of measures having a density in the Sobolev space $W^{1,1}$. In this case $A$ will be a subset of this space. If this set is bounded for the Sobolev norm, then it has bounded Lipschitz multipliers. 

With the above definitions we can state the main general result of the
paper, linking the scaling behavior of the hitting time  of typical orbits and the local
dimension of $\mu $.

%If the two above definitions are satisfied for a certain %system, then one
%can verify the main assumptions of Theorem \ref{main}.

\begin{theorem}\label{main}
Let us consider a  probability measure $\mu _{0}\in PM(X)$, suppose that the set $$A:=\{\mu _{k}:=L^{(k)}\mu_0,k\in 
\mathbb{N}\}$$ is bounded in $B_s$ and has uniformly bounded Lipschitz multipliers. Suppose furthermore that $(X,T_{i})$ has
convergence to equilibrium with superpolynomial speed as in Definition \ref{def1}.
 Suppose $y\in X$ is such that the local dimension $d_{\mu }(y)$ \
of $\mu $ at $y$ exists in the sense of \eqref{dim} and also suppose that the preimages of $y$ have zero $\mu_0$ measure: more precisely let us suppose that $\forall i\in \mathbb{N} $ 
\begin{equation}\label{preimage}
\mu_0 (\{x \ s.t. \ T^{(i)}(x)=y\})=0.
\end{equation}
Then we have
\begin{equation*}
\underset{r\rightarrow 0}{\lim }\frac{\log \tau _{r}(x,y)}{-\log r}%
=d_{\mu }(y)
\end{equation*}%
 for $\mu _{0}$ almost every $x$.
\end{theorem}

\begin{remark}
    We remark that the assumption \eqref{preimage} is automatically satisfied if  the maps considered have countable degree (that is $\forall x \in X, $ the set $T^{-1}(x)$ is countable)
and $\mu_0$ is nonatomic.
\end{remark}

In order to prove the main result we need some preliminary result.

The first one is a kind of
dynamical Borel-Cantelli Lemma adapted to our case.

%\begin{definition}
%We say that the iterates $\mu _{n}=L_{T^{(n)}}\mu _{0}$ of $\mu %_{0}$
%converge to $\mu $ with a superpolynomial speed in the W %distance if there
%is a decreasing function $\Phi (n):\mathbb{N\rightarrow R}$ %such that \ for
%each $\alpha \in \mathbb{R}$ $\lim_{n\rightarrow \infty %}n^{\alpha }$ $\Phi
%(n)=0$ and 
%\begin{equation*}
%W(\mu _{n},\mu )\leq \Phi (n).
%\end{equation*}
%\end{definition}

%\begin{definition} 
%\ Suppose $x_{0}\in X$ is such that the local dimension $d_{\mu %}(x_{0})$ \
%of $\mu $ at $x_{0}$ exists, then the equality 
%\begin{equation*}
%\underset{r\rightarrow 0}{\lim }\frac{\log \tau _{r}(x,x_{0})}{-%\log r}%
%=d_{\mu }(x_{0})
%\end{equation*}%
%holds for $\mu _{0}$ almost every $x$.
%\end{definition}

\begin{lemma}
\label{unoo}
Let $(X,T_{i})$ be a sequential nonautonomous system, let $\mu
_{0}\in PM(X)$ and $\mu _{k}:=L^{(k)}\mu_0$ as above. 

Suppose there is $\mu \in PM(X)$ and a superpolynomially decreasing $\Phi $ such that for each $g\in Lip(X),$
%with $\sup_{x\in X} |g(x)|\leq 1$ and
$j,k\in \mathbb{%
N}$, the measure $g\mu _{j}$ converges to $\mu \int gd\mu _{j}$ at a uniform
superpolynomial speed in the W distance: more precisely
for each such $g,j,k$
\begin{equation}
||L_{T_{j+k}}\circ ...\circ L_{T_{j+1}}[g\mu _{j}]-\mu \int gd\mu
_{j}||_{W}\leq \max (1,||g||_{Lip})\Phi (k).  \label{sp}
\end{equation}%
Let $g_{k}$ be a sequence of positive Lipschitz observables such
that%
\begin{equation*}
\sup_{x\in X,k\in \mathbb{N}}|g_{k}(x)|\leq 1.
\end{equation*}%
Suppose that $\exists B\geq 1,\ \beta >0$ such that $||g_{k}||_{Lip}\leq Bk^{\beta }$ and
suppose that $\exists \gamma ,C>0$ such that%
\begin{equation}
 \sum_{j\leq n}\int g_{j}(T^{(j)}(x))d\mu _{0} \geq Cn^{\gamma }.%\leq Cn^{\gamma }.
\label{122}
\end{equation}
Then%
\begin{equation*}
\frac{\sum_{j\leq n}g_{j}(T^{(j)}(x))}{\sum_{j\leq n}\int
g_{j}(T^{(j)}(x))d\mu _{0}}\rightarrow 1
\end{equation*}%
$\mu _{0}$ almost everywhere.
\end{lemma}

\begin{proof}
First let us remark that for the Lipschitz observables $g_{j},$ \ by the
fast convergence to equilibrium \eqref{sp} we get that%
\begin{eqnarray}
|\int g_{j}(T^{(j)}(x))d\mu _{0}-\int g_{j}d\mu | &=&|\int g_{j}d\mu
_{j}-\int g_{j}d\mu |  \label{sou} \\
&\leq &||g_{j}||_{Lip}||\mu _{j}-\mu ||_{W}\leq Bj^{\beta }\Phi (j)
\end{eqnarray}%
and since $Bj^{\beta }\Phi (j)$ is summable we get that there is $C_{2}>0$
such that%
\begin{equation*}
 \sum_{j\leq n}\int g_{j}d\mu \geq C_{2}n^{\gamma }.
\end{equation*}

Let $\gamma $ as above, consider $\alpha <\frac{\gamma }{2}$%
\begin{eqnarray*}
\int \left( \sum_{1\leq k\leq n}g_{k}(T^{(k)}(x))\right) ^{2}d\mu _{0}
&=&\sum_{1\leq j\leq n}\int (g_{j}(T^{(j)}(x)))^{2}d\mu _{0} \\
&&+2\sum_{\substack{ k,j\leq n,k>j  \\ k<j+n^{\alpha }}}\int
g_{j}(T^{(j)}(x))g_{k}(T^{(k)}(x))d\mu _{0} \\
&&+2\sum_{\substack{ k,j\leq n  \\ k\geq j+n^{\alpha }}}\int
g_{j}(T^{(j)}(x))g_{k}(T^{(k)}(x))d\mu _{0}.
\end{eqnarray*}%
Since $\forall i$, $0\leq g_{i}\leq 1$ this implies\ $%
g_{j}(T^{(j)}(x))g_{k}(T^{(k)}(x))\leq g_{k}(T^{(k)}(x))$ and 
\begin{eqnarray}
&& \nonumber \sum_{1\leq j\leq n}\int (g_{j}(T^{(j)}(x)))^{2}d\mu _{0}+2\sum_{\substack{
k,j\leq n,k>j  \\ k<j+n^{\alpha }}}\int
g_{j}(T^{(j)}(x))g_{k}(T^{(k)}(x))d\mu _{0}  \label{abv2} \\
&\leq &2n^{\alpha }\sum_{j\leq n}\int g_{j}(T^{(j)}(x))d\mu _{0}.
\end{eqnarray}%
Now let us estimate%
\begin{equation*}
\sum_{k,j\leq n,k\geq j+n^{\alpha }}\int
g_{j}(T^{(j)}(x))g_{k}(T^{(k)}(x))d\mu _{0}.
\end{equation*}%
We have%
\begin{equation*}
|\int g_{j}(T^{(j)}(x))g_{k}(T^{(k)}(x))d\mu _{0}|\leq |\int
g_{k}(x)dL^{(j+1,k)}[g_{j}d\mu _{j}]|
\end{equation*}%
where $L^{(j+1,k)}:=L_{T^{k}}\circ ...\circ L_{T^{j+1}}$.
By \eqref{sp} $$
||L^{(j+1,k)}[g_{j}\mu _{j}]-[\int g_{j}(x)d\mu _{j}]\mu ||_{W}\leq \max
(1,||g_{j}||_{Lip})\Phi (k-j-1)$$ and then%
\begin{eqnarray*}
|\int g_{k}(x)dL^{(j+1,k)}[g_{j}\mu _{j}]| &\leq &\int g_{k}(x)d\mu \int
g_{j}(x)d\mu _{j}+B^2[k]^{\beta }[j+1]^{\beta }\Phi (k-j-1) \\
&\leq &[\int g_{k}\circ T^{(k)}d\mu _{0}+B[k]^{\beta }\Phi (k)]\int
g_{j}\circ T^{(j)}d\mu _{0} \\
&&+B^2[k]^{\beta }[j+1]^{\beta }\Phi (k-j-1) \\
&\leq &\int g_{k}\circ T^{(k)}d\mu _{0}\int g_{j}\circ T^{(j)}d\mu _{0} \\
&&+B[k]^{\beta }\Phi (k)+B^2[k]^{\beta }[j+1]^{\beta }\Phi (k-j-1)
\end{eqnarray*}
using again $(\ref{sou}).$ Hence%
\begin{eqnarray}
\sum_{\substack{ k,j\leq n  \\ k\geq j+n^{\alpha }}}\int g_{j}\circ
T^{(j)}~g_{k}\circ T^{(k)}d\mu _{0} &\leq &\sum_{\substack{ k,j\leq n  \\ %
k\geq j+n^{\alpha }}}[\int g_{j}\circ T^{(j)}d\mu _{0}\int g_{k}\circ
T^{(k)}d\mu _{0}  \label{abv} \\
&& \nonumber +B[k]^{\beta }\Phi (k)+B^2 [k]^{\beta }[j+1]^{\beta }\Phi (k-j-1)] \\
&\leq &\sum_{\substack{ k,j\leq n  \\ k\geq j+n^{\alpha }}}[\int g_{j}\circ
T^{(j)}d\mu _{0}\int g_{k}\circ T^{(k)}d\mu _{0}] \\
&&+2B^2 n^{2\beta +2}\Phi (n^{\alpha })+Bn^{\beta +2}\Phi (n^{\alpha }).
\end{eqnarray}

Now consider the sequence of random variables $Z_{n}(x):=\sum_{1\leq j\leq
n}g_{j}(T^{(j)}(x))$ and denote by $E(Z_{n}):=\int \sum_{1\leq j\leq
n}g_{j}(T^{(j)}(x))d\mu _{0}(x)$ let us consider the additional sequence of
random variables 
\begin{equation*}
Y_{n}=\frac{Z_{n}}{E(Z_{n})}-1=\frac{Z_{n}-E(Z_{n})}{E(Z_{n})}.
\end{equation*}%
And since $\int \left( Z_{n}-E(Z_{n})\right) ^{2}d\mu _{0}=\int \left(
Z_{n}\right) ^{2}d\mu _{0}-(E(Z_{n}))^{2}$ we get%
\begin{eqnarray*}
E((Y_{n})^{2}) &=&\frac{\int Z_{n}^{2}d\mu _{0}-E(Z_{n})^{2}}{E(Z_{n})^{2}}
\\
&=&\frac{\int \left( \sum_{1\leq k\leq n}g_{k}(T^{(k)}(x))\right) ^{2}d\mu
_{0}-(\sum_{1\leq k\leq n}\int g_{k}(T^{(k)}(x))d\mu _{0})^{2}}{(\sum_{1\leq
k\leq n}\int g_{k}(T^{(k)}(x))d\mu _{0})^{2}} \\
&\leq &\frac{ 2n^{\alpha }\sum_{j\leq n}\int g_{j}(T^{(j)}(x))d\mu
_{0}+B^2 4n^{2\beta +2}\Phi (n^{\alpha })+2Bn^{\beta +2}\Phi (n^{\alpha })}{%
(\sum_{1\leq k\leq n}\int g_{k}(T^{(k)}(x))d\mu _{0})^{2}}
\end{eqnarray*}%
where in the last line we used $(\ref{abv})$ and $(\ref{abv2})$. By this and 
$(\ref{122})$, since $\alpha <\frac{\gamma }{2}$ we establish $%
E((Y_{n})^{2})\rightarrow 0.$ Now consider%
\begin{equation}
n_{k}=inf\{{n:\sum_{1\leq j\leq n}\int g_{j}(T^{(j)}(x))d\mu _{0}\geq k^{2}\}%
}.
\end{equation}%
\begin{eqnarray*}
E((Y_{n_{k}})^{2}) &\leq &\frac{2n_{k}^{\alpha }\sum_{j\leq n_{k}}\int
g_{j}(T^{(j)}(x))d\mu _{0}+2B^2n_{k}^{2\beta +2}\Phi (n_{k}^{\alpha
})+2Bn_{k}^{\beta +2}\Phi (n_{k}^{\alpha })}{(\sum_{1\leq j\leq n_{k}}\int
g_{j}(T^{(j)}(x))d\mu _{0})^{2}} \\
&\leq &\frac{2n_{k}^{\alpha }}{\sum_{j\leq n_{k}}\int g_{j}(T^{(j)}(x))d\mu
_{0}}+\frac{4B^2n_{k}^{2\beta +2}\Phi (n_{k}^{\alpha })}{(\sum_{1\leq j\leq
n_{k}}\int g_{j}(T^{(j)}(x))d\mu _{0})^{2}} \\
&&+\frac{2Bn_{k}^{\beta +2}\Phi (n_{k}^{\alpha })}{(\sum_{1\leq j\leq
n_{k}}\int g_{j}(T^{(j)}(x))d\mu _{0})^{2}}
\end{eqnarray*}

Since $\forall \epsilon >0$, for $n$ big enough, $\sum_{j\leq n}\int
g_{j}(T^{(j)}(x))d\mu _{0}\geq n^{\gamma -\epsilon }$ then $n_{k}\leq (k+1)^{%
\frac{2}{\gamma -\epsilon }}\leq (2k)^{\frac{2}{\gamma -\epsilon }}$ and 
\begin{equation*}
\frac{2n_{k}^{\alpha }}{\sum_{j\leq n_{k}}\int g_{j}(T^{(j)}(x))d\mu _{0}}%
\leq \frac{2(2k)^{\frac{2\alpha }{\gamma -\epsilon }}}{k^{2}}
\end{equation*}%
and since $\alpha <\frac{\gamma }{2}$, and $\epsilon $ can be taken small as
wanted, we have that 
\begin{equation*}
\sum_{k\geq 0}E((Y_{n_{k}})^{2})<\infty
\end{equation*}%
then by the classical Borel Cantelli Lemma (See  e.g. \cite{KA06})
$Y_{n_{k}}\rightarrow 0$ a.e. \ Now we
prove that the whole $Y_{n}\rightarrow 0$ a.e. Indeed if $n_{k}\leq n\leq
n_{k+1}$%
\begin{equation*}
\frac{Z_{n}}{E(Z_{n})}\leq \frac{Z_{n_{k+1}}}{E(Z_{n_{k}})}=\frac{Z_{n_{k+1}}%
}{E(Z_{n_{k+1}})}\frac{E(Z_{n_{k+1}})}{E(Z_{n_{k}})}\leq \frac{Z_{n_{k+1}}}{%
E(Z_{n_{k+1}})}\frac{(k+2)^{2}}{k^{2}}
\end{equation*}%
and%
\begin{equation*}
\frac{Z_{n}}{E(Z_{n})}\geq \frac{Z_{n_{k}}}{E(Z_{n_{k+1}})}=\frac{Z_{n_{k}}}{%
E(Z_{n_{k}})}\frac{E(Z_{n_{k}})}{E(Z_{n_{k+1}})}\geq \frac{Z_{n_{k}}}{%
E(Z_{n_{k}})}\frac{k^{2}}{(k+2)^{2}}.
\end{equation*}%
then we have $\underset{n\rightarrow \infty }{\lim }\frac{Z_{n}}{E(Z_{n})}%
=1, $ $\mu -$almost everywhere.
\end{proof}

We will use the last Lemma to prove a proposition which will be an intermediate step in proving Theorem \ref{main}. 

%\begin{definition}
%We say that the iterates $\mu _{n}=L_{T^{(n)}}\mu _{0}$ of $\mu %_{0}$
%converge to $\mu $ with a superpolynomial speed in the W %distance if there
%is a decreasing function $\Phi (n):\mathbb{N\rightarrow R}$ such that \ for
%each $\alpha \in \mathbb{R}$ $\lim_{n\rightarrow \infty %}n^{\alpha }$ $\Phi
%5(n)=0$ and 
%\begin{equation*}
%W(\mu _{n},\mu )\leq \Phi (n).
%\end{equation*}
%\end{definition}

\begin{proposition}
\label{mainprop}Let $(X,T_{i})$ be a sequential nonautonomous system, let $\mu
_{0}\in PM(X)$ as above and  suppose there is $\mu \in PM(X)$ such that for each $g\in Lip(X)$, every $g\mu_i$ converges uniformly to $\mu$  at a superpolynomial speed as in \eqref{sp}. Let us consider a target point $y\in X$ such that assumption \eqref{preimage} is satisfied.
Then the equality 
\begin{equation*}
\underset{r\rightarrow 0}{\lim }\frac{\log \tau _{r}(x,y)}{-\log r}%
=d_{\mu }(y)
\end{equation*}%
holds for $\mu _{0}$ almost every $x$.
\end{proposition}

This proposition is of independent interest since directly establishes the logarithm law. However the assumption required (see \eqref{sp}) may look quite technical and difficult to verify. For this we decided to state the main result in the form shown at Theorem \ref{main}, whose assumptions are more similar to what is expected to be  established in concrete examples, where the maps involved are already known to satisfy some regularization (Lasota Yorke) inequalities on certain functional spaces, and some convergence to equilibrium properties.

%\begin{remark}
%Condition $(\ref{sp})$ might look quite technical, but we will %see in Lemma %
%\ref{bomu} \ that this condition is verified by some general %condition on
%the regularity of the measures $\mu _{j}$, which can be %verified by suitable
%sequential Lasota Yorke inequalities.
%\end{remark}

\begin{proof}[Proof of Proposition \protect\ref{mainprop}]
Let $r_{k}\rightarrow 0$ be a decreasing sequence, let $B(y,r_{k})$ be a
sequence of balls with decreasing radius centered at $y$, \ let $\phi
_{k}$ be a Lipschitz function such that $\phi _{k}(x)=1$ for all $x\in
B(y,r_{k})$, $\phi _{k}(x)=0$ if $x\notin B(y,r_{k-1})$ and $||\phi
_{k}||_{Lip}\leq \frac{1}{r_{k-1}-r_{k}}$ (such functions can be constructed
as $\phi _{k}(x)=h(d(y,x))$ where $h$ is a piecewise linear Lipschitz
function $\mathbb{R\rightarrow }[0,1]$).

First we prove that $\underset{r\rightarrow 0}{\lim \inf }\frac{\log \tau
_{r}(x,y)}{r}\geq d_{\mu }(y)$, $\mu _{0}$ almost everywhere. This
follows by a classical Borel-Cantelli argument.

Let us suppose that for some $\epsilon >0,$ $\underset{r\rightarrow 0}{\lim
\inf }\frac{\log \tau _{r}(x,y)}{-\log r}\leq d_{\mu }(y)-\epsilon $
on \ a certain set $A\subseteq X$. \ Let us consider the sequence $%
r_{k}=k^{-(d_{\mu }(y)-\epsilon )^{-1}}$. From the properties of
logarithms,\ it is standard to get (see the beginning of the proof of
Theorem 4 of \cite{G}) $\underset{r\rightarrow 0}{\lim \inf }\frac{\log \tau
_{r}(x,y)}{-\log r}=\underset{k\rightarrow \infty }{\lim \inf }\frac{%
\log \tau _{r_{k}}(x,y)}{-\log r_{k}}.$ Hence $\underset{r\rightarrow 0}{%
\lim \inf }\frac{\log \tau _{r}(x,y)}{-\log r}\leq d_{\mu
}(y)-\epsilon $ implies that $\frac{\log \tau _{r_{k}}(x,y)}{-\log
r_{k}}\leq d_{\mu }(y)-\epsilon $ for infinitely many $k$'s.

We have that for each $\epsilon ^{\prime }>0$ eventually when $k$ is large enough $$\int \phi
_{k} d\mu \leq (k-1)^{-(d_{\mu }(y)-\epsilon )^{-1}d_{\mu }(y)-\epsilon
^{\prime }}.$$ If $\epsilon ^{\prime }$ is so small that $(d_{\mu
}(y)-\epsilon )^{-1}d_{\mu }(y)-\epsilon ^{\prime }>1,~$then $%
\sum_{k}\int \phi _{k}d\mu <\infty .$ \ Let us now consider the sequence $%
\phi _{k}\circ T^{(k)}$ and let us estimate $\int \phi _{k}\circ
T^{(k)} d\mu _{0}.$ \ We have eventually as $k\rightarrow \infty $\ that%
\begin{eqnarray*}
|\int \phi _{k}\circ T^{(k)}d\mu _{0}-\int \phi _{k}d\mu | &=&|\int \phi
_{k}d\mu _{k}-\int \phi _{k}d\mu | \\
&\leq &||\phi _{k}||_{Lip}||\mu _{k}-\mu ||_{W} \\
&\leq &k^{\beta }\Phi (k)
\end{eqnarray*}%
where $\beta >\lim_{k\rightarrow \infty }\frac{\log \frac{1}{r_{k-1}-r_{k}}}{%
\log k}$, and since $k^{\beta }\Phi (k)$ is summable we get that for each
such $\epsilon >0$%
\begin{equation*}
\sum_{k\leq n}\int \phi _{k}\circ T^{(k)}d\mu _{0}<\infty .
\end{equation*}%
This means that the set of $x\in X$ for which $\sum_{k}\phi
_{k}(T^{(k)}(x))=\infty $ is a zero $\mu _{0}$-measure set. This set
includes the set of $x$ such that $d(T^{(k)}(x),y)\leq r_{k}$ infinitely
many times  and \ the set of $x\in X$ such that $\forall i$ $T^{(i)}(x)\neq y$ and for infinitely many $k$, $\tau_{r_{k}}(x,y)\leq k=r_{k}^{-(d_{\mu }(y)-\epsilon )}$ proving that $A$ is a zero $\mu _{0}$-measure one \footnote{If $\forall i$ $T^{(i)}(x)\neq y$ and $\tau_{r_{k}}(x,y)\leq k$ for infinitely many $k$, we have infinitely many $k$ for which $d(T^{(k)}(x),y)\leq r_k$. Indeed assuming the opposite. Let us consider $k$ to be the last index for which  $d(T^{(k)}(x),y)\leq r_k$. Since $min_{i\leq k}d(T^{(i)}(x),y)>0$ we can  consider $k'>k$ such that $0<r_{k'}<min_{i\leq k}d(T^{(i)}(x),y)$ since still we have   $\tau_{r_{k'}}(x,y)\leq k'$ we have a new close approach to the target, negating the asumption.}.

Now we prove that 
\begin{equation}
\underset{r\rightarrow 0}{\lim \sup }\frac{\log \tau _{r}(x,y)}{r}\leq
d_{\mu }(y)  \label{sec}
\end{equation}%
$\mu _{0}$ almost everywhere. Let us consider some small $\epsilon ^{\prime
}>0$ and the set of $x$ such that $\underset{r\rightarrow 0}{\lim \sup }%
\frac{\log \tau _{r}(x,y)}{r}\geq d_{\mu }(y)+\epsilon ^{\prime }.$
In order to estimate the measure of such set, let us consider some $0<\beta <%
\frac{1}{d_{\mu }(y)}$ (implying $\beta d_{\mu }(y)<1$) \ such that $%
\beta (d_{\mu }(y)+\epsilon ^{\prime })>1$. Consider then the sequence
of radii $r_{k}=k^{-\beta }$. We remark that as before, if $(\ref{sec})$ is
proved for such a sequence, then it holds for all sequences. Now remark that
for each small $\epsilon <\beta ^{-1}-d_{\mu }(y)$, eventually as $k\to \infty$, $\int
\phi _{k}d\mu \geq (r_{k})^{d_{\mu }(y)+\epsilon }=k^{-\beta (d_{\mu
}(y)+\epsilon )}$ and there is $C>0$ such that 
\begin{equation}
\sum_{0}^{k}\int \phi _{k}d\mu \geq Ck^{1-\beta (d_{\mu }(y)+\epsilon )}
\label{232}
\end{equation}%
for each $k\geq 0.$ Now let us estimate the sequence $\int \phi _{k}\circ
T^{(k)}d\mu _{0}.$ We have as before that eventually, as $k\rightarrow
\infty $%
\begin{eqnarray*}
|\int \phi _{k}\circ T^{(k)}d\mu _{0}-\int \phi _{k}d\mu | &=&|\int \phi
_{k}d\mu _{k}-\int \phi _{k}d\mu | \\
&\leq &||\phi _{k}||_{Lip}||\mu _{k}-\mu ||_{W} \\
&\leq &k^{\beta ^{\prime }}\Phi (n)
\end{eqnarray*}%
for some $\beta ^{\prime }>\lim_{k\rightarrow \infty }\frac{\log \frac{1}{%
r_{k-1}-r_{k}}}{\log k}$ since $k^{\beta ^{\prime }}\Phi (n)$ is summable we
get by $(\ref{232})$ that for each $\epsilon >0, $ eventually 
\begin{equation*}
\sum_{k\leq n}\int \phi _{k}\circ T^{(k)}d\mu _{0}\geq n^{1-\beta d_{\mu
}(y)- \beta \epsilon }.
\end{equation*}
We can then apply Lemma \ref{unoo} and obtain that setting $Z_{n}(x)=\sum_{j\leq n}\phi _{j}(T^{(j)}(x))$
and $E(Z_{n})=\int \sum_{j\leq n}\phi _{j}(T^{(j)}(x))d\mu _{0}$, for such a
sequence, $\underset{n\rightarrow \infty }{\lim }\frac{Z_{n}}{E(Z_{n})}=1$ $%
\mu _{0}-$almost everywhere. We are now going to use this to complete the
proof. Let us hence still consider  $\beta $ as above, near but below $\frac{1%
}{d_{\mu }(y)}$ and $\epsilon ^{\prime }>0$ such that $\beta (d_{\mu
}(y)+\epsilon ^{\prime })>1$. Let us consider $x$ such that $\underset{%
r\rightarrow 0}{\lim \sup }\frac{\log \tau _{r}(x,y)}{r}\geq d_{\mu }(y)+\epsilon ^{\prime }$
then, for infinitely many $n$, $\tau _{(n-1)^{-\beta
}}(x,y)\geq (n-1)^{\beta (d_{\mu }(y)+\epsilon ^{\prime })},$ then $%
T^{(i)}(x)\notin B(y,(n-1)^{-\beta })$ \ \ for each $0\leq i\leq (n-1)^{\beta
(d_{\mu }(y)+\epsilon ^{\prime })}$ and in particular $%
T^{(i)}(x)\notin B(y,(i-1)^{-\beta })$ for $n\leq i\leq (n-1)^{\beta (d_{\mu
}(y)+\epsilon ^{\prime })}$ \ which implies $Z_{n}(x)=Z_{n^{\beta
(d_{\mu }(y)+\epsilon ^{\prime })}}(x)$ for infinitely many $n.$ But 
\begin{equation*}
\frac{\sum_{i=0}^{n^{\beta (d_{\mu }(y)+\epsilon ^{\prime })}}\int
\phi _{j}d\mu }{\sum_{i=0}^{n}\int \phi _{j}d\mu }\geq \frac{%
\sum_{i=0}^{n^{\beta (d_{\mu }(y)+\epsilon ^{\prime })}}\mu
(B(y,i^{-\beta }))}{\sum_{i=0}^{n}\mu (B(y,(i-1)^{-\beta }))}%
\rightarrow \infty
\end{equation*}%
eventually as $n\rightarrow \infty $ because $\beta d_{\mu }(y)<1$, implying that the above sums go to $\infty$ and because $\beta
(d_{\mu }(y)+\epsilon ^{\prime })>1$, implying that the numerator's sum goes to $\infty$ faster than the denominator's one. Then as shown before%
\begin{equation*}
\frac{E(Z_{n^{\beta (d_{\mu }(y)+\epsilon ^{\prime })}})}{E(Z_{n})}=%
\frac{\sum_{j=0}^{n^{\beta (d_{\mu }(y)+\epsilon ^{\prime })}}\int
\phi _{j}(T^{(j)}(x))d\mu _{0}}{\sum_{j=0}^{n}\int \phi
_{j}(T^{(j)}(x))d\mu _{0}}\rightarrow \infty
\end{equation*}%
hence in order to get $\underset{n\rightarrow \infty }{\lim }\frac{Z_{n}}{%
E(Z_{n})}=1$ $\ \mu _{0}$-almost everywhere\ one must have $\underset{%
r\rightarrow 0}{\lim \sup }\frac{\log \tau _{r}(x,y)}{r}\geq d_{\mu
}(y)+\epsilon ^{\prime }$ on a zero measure set. Since $\beta $ can be
chosen as near as we want to $\frac{1}{d_{\mu }(y)}$, $\epsilon ^{\prime
}$ can be chosen to be arbitrary small we have the statement.
\end{proof}

Now we see that the assumptions of Theorem \ref{main} implies the ones of Proposition \ref{mainprop} and then we can get our main result applying the proposition.

\begin{lemma}
\label{bomu}Given a probability measure $\mu _{0}\in B_{s}$, let us suppose that the set $A:=\{\mu _{k}:=L^{(k)}\mu_0,k\in 
\mathbb{N}\}$ is bounded in $B_s$ and has uniformly bounded Lipschitz multipliers in the sense of Definition \ref{def1}. Suppose  $(X,T_{i})$ has
convergence to equilibrium with superpolynomial speed in the sense of Definition \ref{def2} and there is $C\geq 0$
such that $||~||_{W}\leq C||~||_{s}$, then  for each $g\in
Lip(X),$ $i\in \mathbb{N}$ the measure $g\mu _{i}$ converges to $\mu \int
gd\mu _{i}$ at a superpolynomial speed as expressed in $(\ref{sp}).$
\end{lemma}

\begin{proof}
Since $A$ is bounded  there is a $C_{2}\geq 0$ s.t. \ $||\mu _{j}||_{s}\leq
C_{2}$ $\forall j$ and $||\mu ||_{s}\leq C_{2}.$
In order to prove $(\ref{sp})$, from the convergence to equilibrium we have 
\begin{eqnarray*}
||L_{T_{j+k}}\circ ...\circ L_{T_{j+1}}[g\mu _{j}]-\mu \int gd\mu
_{j}||_{W} &\leq &\Phi (k) max(1,||[g\mu _{j}]-\mu \int gd\mu _{j}||_{s}).
\end{eqnarray*}

By the bounded Lipschitz multiplier property $$||[g\mu _{j}]-\mu \int gd\mu _{j}||_{s}\leq C_{A}||g||_{Lip}+C_{2}||g||_{Lip}$$ since $||g||_{\infty }\leq ||g||_{Lip}$ 
and $\mu
_{j} $ is a probability measure. We have then 
\begin{equation*}
||L_{T_{j+k}}\circ ...\circ L_{T_{j+1}}[g\mu _{j}]-\mu \int g_{}d\mu
_{j}||_{W}\leq \Phi'(k) \max(1,||g||_{Lip})
\end{equation*}%
for a superpolinomially decreasing $\Phi'$
as required by $(\ref{sp}).$
\end{proof}

Having collected the necessary results, we can now prove the main theorem.

\begin{proof}[Proof of Theorem \ref{main}]
By Lemma \ref{bomu} we see that the assumptions of Theorem \ref{main} imply the assumptions of Proposition \ref{mainprop}.
The application of this proposition directly lead to the result.
\end{proof}

\section{Application to asymptotically autonomous systems\label{yp}}
In this section we show an example of application of Theorem \ref{main} to a family of solenoidal  maps forming a nonautonomous system.
Such a family is also an eventually autonomous system in the sense of \cite{Ash}.
%The study of such kind of nonautonomous behavior is claimed to have important implications in the study of climate change  (see \cite{Ash}, \cite{Ash2}). 
Solenoidal maps are known to have a fractal attractor whose dimension can vary, depending on the map's contraction and expansion rates (see Figure 1).

\begin{figure}[h]
    \centering
    \includegraphics[width=0.35\textwidth]{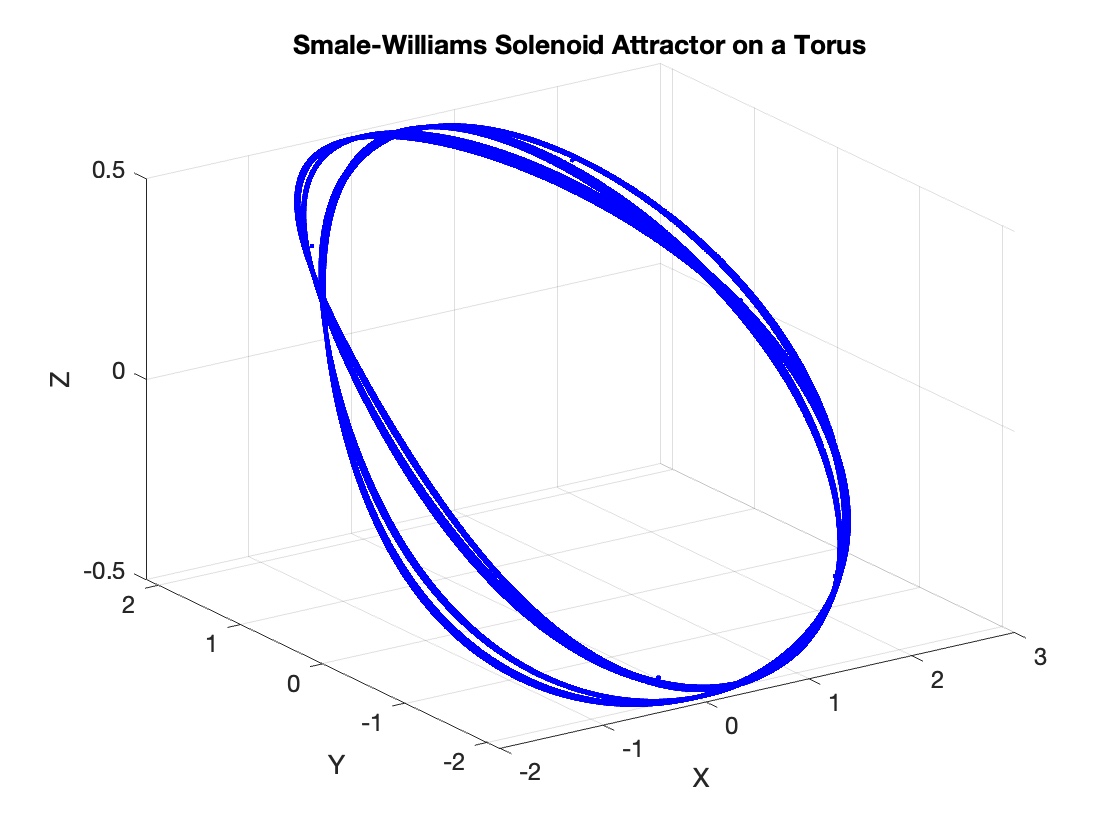}
    \caption{Example of a typical attractor in the phase space $\mathbb{S}^{1}\times D^{2}$ for a solenoidal map.}
\end{figure}

The same can be said for the local dimension of the unique physical invariant measure. To keep the treatment short and avoid technicalities, we choose a relatively simple family of such maps where the maps vary in time only on the second coordinate.
We therefore consider a family $F_i$ of solenoidal maps. Each element of $F_i$ is a $C^{2}$ map $F_i:X\rightarrow X$ where $X=\mathbb{S}^{1}\times
D^{2}$ the filled torus, and $F_i$ is a skew product 
\begin{equation}
F_i(x,y)=(T(x),G_i(x,y)),  \label{1eq}
\end{equation}%
where $T: \mathbb{S}^{1} \longrightarrow \mathbb{S}^{1}$ and $G_i:X\longrightarrow D^{2}$ are smooth maps. We suppose the map $T:\mathbb{S}^{1}\rightarrow \mathbb{S}^{1}$ to be $C^{3}$, expanding \footnote{There is $\alpha 
 <1$ such that $\forall x \in  \mathbb{S}^{1}$, $|T_{0}^{\prime }(x)|\geq \alpha ^{-1}>1$.} of degree $q$, giving rise to a map $[0,1]\rightarrow
\lbrack 0,1]$, which we denote by $\Tilde{T}$ \ and
whose branches will be denoted by $\Tilde{T}_{i}$, $i\in \lbrack 1,..,q]$. We
make the following assumptions on the $G_i:$

\begin{description}
\item[(a)] Consider the $F$-invariant foliation $\mathcal{F}^{s}:=\{\{x\}\times
D^{2}\}_{x\in S^{1}}$. We suppose that $\mathcal{F}^{s}$ is contracted:
there exists $0<\alpha <1$ such that for all $x\in \mathbb{S}^{1}, i\in \mathbb{N}$ 
\begin{equation}
|G_i(x,y_{1})-G_i(x,y_{2})|\leq \alpha |y_{1}-y_{2}|\ \ \mathnormal{for\ all}\ \
y_{1},y_{2}\in D^{2}.  \label{contracting1}
\end{equation}

\item[(b)] $\sup_{x\in S^1,i\in \mathbb{N}}|\frac{\partial G_i}{\partial x}(x)|<\infty .$

\item[(c)] $\sup_{x,y}|G_i(x,y)-G_0(x,y)| \leq \Phi(i)$ with $\Phi$ being decreasing and having superpolynomial decay.
\end{description}

In the following, applying the theory we have shown in the previous sections, we will prove the logarithm law for this system:

\begin{proposition}\label{solw}
Let $(X,F_{i})$ be a sequential family of solenoidal maps satisfying the above assumptions $(a),(b),(c)$, let $\mu_0$ be the Lebesgue measure on $X$ and $\mu$ the unique physical measure of $F_0$.

Suppose $y\in X$ is such that the local dimension $d_{\mu }(y)$ \
of $\mu $ at $y$ exists, then the equality 
\begin{equation*}
\underset{r\rightarrow 0}{\lim }\frac{\log \tau _{r}(x,y)}{-\log r}%
=d_{\mu }(y)
\end{equation*}%
holds for $\mu _{0}$ almost every $x$.
   
\end{proposition}
We will prove Proposition \ref{solw} by applying Theorem  \ref{main}.
In order to do this, we construct some functional spaces which are suitable for the system we
consider. 
Functional spaces adapted to uniformly hyperbolic systems like solenoidal maps have been studied in \cite{GL06}.
Here we use a simpler construction of anisotropic spaces suitable for skew products which can be found in \cite{GalJEP} and \cite{GalLuc}. The idea is to consider spaces of measures with sign, with
suitable norms constructed by disintegrating the measures along the stable, preserved
foliation. 

 Given $\mu \in
SM(X)$ denote by $\mu ^{+}$ and $\mu ^{-}$ the positive and the negative
parts of it ($\mu =\mu ^{+}-\mu ^{-}$).

Let  $\pi _{x}:X\longrightarrow \mathbb{S}^{1}$ be the projection defined by $\pi
(x,y)=x$ and let $\pi _{x}^{\ast }$ be the associated pushforward map.

Denote by $\mathcal{AB}$ the set of  measures $\mu \in SM(X)$ such
that its associated marginal measures, $\mu _{x}^{+ }:=\pi
_{x}^{\ast }\mu ^{+ }$, $\mu _{x}^{-}:=\pi
_{x}^{\ast }\mu ^{- }$ are absolutely continuous with respect to the
Lebesgue measure $m$ on $\mathbb{S}^{1}$ i.e.%
\begin{equation}
\mathcal{AB}=\{\mu \in SM(X):\pi _{x}^{\ast }\mu ^{+}<<m\ \ \mathnormal{and}%
\ \ \pi _{x}^{\ast }\mu ^{-}<<m\}.  \label{thespace1}
\end{equation}% 
Let us consider a finite positive measure $\mu \in \mathcal{AB}$ on the
space $X$ foliated by the contracting leaves $\mathcal{F}^{s}=\{\gamma
_{l}\}_{l\in \mathbb{S}^{1}}$ such that $\gamma _{l}={\pi _{x}}^{-1}(l)$. The Rokhlin
Disintegration Theorem describes a disintegration $\left( \{\mu _{\gamma_l
}\}_{\gamma_l \in {\cal{F}}^s },\mu _{x}=:\phi _{\mu}m\right) $ by a family $\{\mu _{\gamma
}\}$ of probability measures on the stable leaves and a non
negative marginal density $\phi _{\mu}:\mathbb{S}^{1}\longrightarrow \mathbb{R}$ with $%
||\phi _{\mu}||_{1}=\mu (X)$.
 By this  disintegration, for each
measurable set $E\subset X$, with the above notations it holds 
\begin{equation}
\mu (E)=\int_{S^{1}}{\mu _{\gamma_l }(E\cap \gamma_l )}d\mu _{x}(l ).
\end{equation}%
\label{rmkv}

\begin{definition}
Let  $\pi _{y}:X\longrightarrow D^{2}$ be the projection
defined by $\pi _{y}(x,y)=y$ and $\gamma \in \mathcal{F}^{s}$. Given a
positive measure $\mu \in \mathcal{AB}$ and its disintegration along the
stable leaves $\mathcal{F}^{s}$, $\left( \{\mu _{\gamma_l }\}_{\gamma_l },\mu
_{x}=\phi _{\mu}m_{}\right) $  we define the \textbf{restriction of $\mu $ on $\gamma_l $} as the
positive measure $\mu |_{\gamma_l }$ on $D^{2}$ (not on the leaf $\gamma_l $)
defined, for all mensurable set $A\subset D^{2}$, by 
\begin{equation*}
\mu |_{\gamma_l }:=\pi _{y}^{\ast }(\phi _{\mu}(l )\mu _{\gamma_l
}).
\end{equation*}%
For a given signed measure $\mu \in \mathcal{AB}$ and its decomposition $\mu
=\mu ^{+}-\mu ^{-}$, define the \textbf{restriction of $\mu $ on $\gamma_l $}
by%
\begin{equation}
\mu |_{\gamma_l }:=\mu ^{+}|_{\gamma_l }-\mu ^{-}|_{\gamma_l }.
\end{equation}%
Similarly we define the marginal density of $\mu $ as
$$\phi_\mu:=\phi_{\mu^+}-\phi_{\mu^-}.$$
\label{restrictionmeasure}
\end{definition}

Now we define a $L^1$ like space of disintegrated measures.

\begin{definition}
\label{l1likespace}Let $\mathcal{L}^{1}\subseteq \mathcal{AB}$ be defined as%
\begin{equation}
\mathcal{L}^{1}:=\left\{ \mu \in \mathcal{AB}:\int_{{\mathbb{S}}^{1}}{W(\mu
^{+}|_{\gamma_l },\mu ^{-}|_{\gamma_l })}dm_{}(l )<\infty \right\}
\label{L1measurewithsign}
\end{equation}%
and define a norm on this space,  $||\cdot ||_{''1''}:\mathcal{L}^{1}\longrightarrow 
\mathbb{R}$, by%
\begin{equation}
||\mu ||_{''1''}=\int_{{\mathbb{S}}^{1}}{W(\mu ^{+}|_{\gamma_l },\mu ^{-}|_{\gamma_l
})}dm_{1}(l ).  \label{l1normsm}
\end{equation}%
\end{definition}

Let us now consider the transfer operator $L_{F}$ associated with $F$. There
is a nice characterization of the transfer operator in our case, showing that this operator works similarly to a one dimensional transfer operator. For the proof see 
\cite{GalLuc} .

\begin{proposition}
For each leaf $\gamma \in \mathcal{F}^{s}$, let us define the map $F_{\gamma
}:D_{2}\longrightarrow D_{2}$ by 
\begin{equation*}
F_{\gamma }=\pi _{y}\circ F|_{\gamma }\circ [\pi _{y}|_\gamma]^{-1}.
\end{equation*}%
For all $\mu \in \mathcal{L}^{1}$ and for almost all $l\in \mathbb{S}^1$  the following
holds 
\begin{equation}
(L_{F}\mu )|_{\gamma_l }=\sum_{i=1}^{q}{\dfrac{F_{\gamma_{T_{i}^{-1}(l )}}^{\ast
}(\mu |_{\gamma_{T_{i}^{-1}(l )}})}{|T_{i}^{^{\prime }}\circ T_{i}^{-1}(l )|}%
}.  \label{niceformulaa}
\end{equation}%
\label{niceformulaab}
\end{proposition}
Here, again $F_\gamma^*$ stands for the pushforward of $F_\gamma$. 
%\begin{remark}
%\label{nice}If \ $F$ is a weak contraction %$:X\rightarrow X$, where $X$ is a
%metric space, for every  $\mu \in SM(X) $ it holds%
%\begin{equation*}
%||L_{F}\mu ||_{W}\leq ||\mu ||_{W}.
%\end{equation*}%
%Indeed, since $F$ is a contraction, if $|g|_{\infty %}\leq 1$ and $Lip(g)\leq
%1$ the same holds for $g\circ F$. Then%
%\begin{eqnarray*}
%\left\vert \int {g~}d(L_{F}(\mu ))\right\vert %&=&\left\vert \int g\circ F{~}%
%d\mu \right\vert \\
%&\leq &\left\vert \left\vert \mu \right\vert %\right\vert _{W}.
%\end{eqnarray*}%
%Taking the supremum over $|g|_{\infty }\leq 1$ and %$Lip(g)\leq 1$ we finish
%the proof of the inequality.
%\end{remark}

In \cite{Gdisp}, Section 12, for a solenoidal map $F$ as defined in this section the following elementary facts are proved.

\begin{proposition}[The weak norm is weakly contracted by $L_{F}$]
If $\mu \in \mathcal{L}^{1}$ then 
\begin{equation}
||L_{F}\mu ||_{''1''}\leq ||\mu ||_{''1''}.
\end{equation}%
\label{weakcontral11234}
\end{proposition}

\begin{proposition}
\label{5.6} For all  $\mu \in \mathcal{L}^{1}$ it holds 
\begin{equation}
||L_{F}\mu ||_{''1''}\leq \alpha ||\mu ||_{''1''}+(\alpha +1)||\phi
_{\mu}||_{1}.  \label{abovv}
\end{equation}
\end{proposition}
We denote by $\mathcal{V} \subseteq \mathcal{L}^1$ the set of measures having $0$ average, i.e. $$\mathcal{V}:=\{\mu \in \mathcal{L}^1|\mu(X)=0\} .$$
\begin{proposition}[Exponential convergence to equilibrium]
\label{5.8} There exist $D\in \mathbb{R}$ and $0<\beta _{1}<1$ such that,
for every signed measure $\mu \in \mathcal{V}$, it holds 
\begin{equation*}
||L_{F}^{n}\mu ||_{''1''}\leq D_{2}\beta _{1}^{n}(||\mu ||_{''1''}+||\phi
_{\mu}||_{_{W^{1,1}}})
\end{equation*}%
for all $n\geq 1$. \ 
\end{proposition}
In the previous proposition $||\ ||_{W^{1,1}}$ stands for the $1,1$ Sobolev norm.
Furthermore the system has an unique invariant measure in $\mathcal{L}^{1}$.
\begin{proposition}
There is a unique $\mu \in \mathcal{L}^{1}$ such that $L_{F}\mu =\mu $.
\end{proposition}

Let us now consider the following stronger norm%
\begin{equation*}
||\mu ||_{s}=||\mu ||_{''1''}+||\phi _{\mu}||_{_{W^{1,1}}}.
\end{equation*}
We can then define
\begin{equation}\label{BS}
    B_{s}:=\{\mu \in \mathcal{L}^1, \ s.t. \ ||\mu ||_{s}<+\infty \}.
\end{equation}

Considering $|| \  ||_{s}$ as a strong norm and   $|| \  ||_{''1''}$ as a weak norm we can easily prove a Lasota Yorke
inequality, showing that the system has a kind of regularization  for these two norms.
\begin{lemma} \label{17} For each $\mu\in B_s$
    \begin{equation}\label{30}
       ||L_{F}\mu ||_{s} \leq \max (\alpha ,\lambda )||\mu ||_{s}+[(\alpha +1)+b]||\phi _{\mu}||_{1}.
    \end{equation}
\end{lemma}

\begin{proof}
By Proposition \ref{5.6} and the Lasota Yorke inequality for expanding maps
\begin{eqnarray*}
||L_{F}\mu ||_{s} &\leq &||L_{F}\mu ||_{''1''}+||L_{T}\phi _{\mu
}||_{_{W^{1,1}}} \\
&\leq &\alpha ||\mu ||_{''1''}+(\alpha +1)||\phi _{\mu}||_{1} \\
&&+\lambda ||\phi _{\mu}||_{_{W^{1,1}}}+B||\phi _{\mu}||_{1} \\
&\leq &\max (\alpha ,\lambda )||\mu ||_{s}+[(\alpha +1)+b]||\phi _{\mu}||_{1}.
\end{eqnarray*}
    
\end{proof}

\begin{remark}
From \eqref{30}, since $||\phi _{\mu}||_{1}\leq||\mu||_{''1''}$ one also can deduce
    \begin{equation}
       ||L_{F}\mu ||_{s} \leq \max (\alpha ,\lambda )||\mu ||_{s}+[(\alpha +1)+b]||\mu||_{''1''}.
    \end{equation}
\end{remark}

We are then going to apply Theorem \ref{main} considering  $(B_{s},||~||_{s}) 
$ as a strong space, as just defined, we will also use $(\mathcal{L}^{1},||~||_{''1''})$ as a  weak
space.  To apply Theorem \ref{main}
we have to verify that the iterates of the Lebesgue measure have bounded
Lipschitz multipliers.

In order to achieve this we need we need to recall some further results on
the regularity of the iterates of measures by solenoidal maps.

Given $\mu \in \mathcal{L}^{1}$ and its marginal density ${{\phi _{\mu }}}$.  Let us consider the following stronger space of measures
\begin{equation*}
''{W^{1,1}}''=\left\{ 
\begin{array}{c}
\mu \in \mathcal{L}^{1} \ s.t. \ {{\phi _{\mu }\in W}}^{1,1}~and~\forall l_{1}~\lim_{l \rightarrow  l_{1}}||\mu |_{\gamma _{l}}-\mu
|_{\gamma _{l_1}}||_{W}=0~and~ \\ 
for~almost~all~l_1,~D(\mu ,l_1):=\limsup_{l \to
l_{1}}||\frac{\mu |_{\gamma _{l_1}}-\mu |_{\gamma _{l}}}{%
l_1-l}||_{W}~<\infty ~and 
\\
||\mu ||_{''1''}+\int |D(\mu ,\gamma_l )|dl ~<\infty
\end{array}%
\right\}.
\end{equation*}

\begin{definition}
\label{W111}Let us consider the norm
\begin{equation*}
||\mu ||_{{{''{W^{1,1}}''}}}:=||\mu ||_{''1''}+\int |D(\mu ,\gamma )|d\gamma .
\end{equation*}
\end{definition}
The following is proved in  \cite{Gdisp}, Section 12. 

\begin{proposition}
\label{LYYY}Let $F$ be a solenoidal map satisfying $(a),(b),(c)$, then $L_{F}({{''{W^{1,1}}''}})\subseteq
{{''{W^{1,1}}''}}$ and there are $\lambda <1,B>0$ s.t $\forall \mu \in {{''{W^{1,1}}''}}$
with $\mu \geq 0$%
\begin{equation*}
||L_{F}\mu ||_{{{''{W^{1,1}}''}}}\leq \lambda ({{\alpha ||\mu ||_{{{''{W^{1,1}}''}}}+||\phi }%
^{\prime }{_{\mu }||}}_{1}{)+}B||\mu ||_{''1''}.
\end{equation*}
\end{proposition}

Iterating  the inequality, one gets

\begin{corollary}\label{trima}
\label{LYYYYYY}There are $B>0,\lambda <1$ such that%
\begin{equation}\label{tri1}
||L^{(n)}\mu ||_{{{''{W^{1,1}}''}}}\leq \lambda ^{n}({{||\mu ||_{{{''{W^{1,1}}''}}}+||\phi 
}^{\prime }{_{\mu }||}_{1})+}B||\mu ||_{1}.
\end{equation}
Where $L^{(n)}$ stands for the sequential composition of the operators $L_{F_i}$ as defined in Section \ref{sec2}.
\end{corollary}

\begin{proof}
By Propositions \ref{weakcontral11234} and \ref{LYYY} the operators $L_i:=L_{F_i }$ satisfy a common
Lasota Yorke inequality. Denoting $||\mu||_s:=||\mu||_{W^{1,1}}+||\phi'_\mu||_1$, there are constants $B,\lambda
_{1}\geq 0$ with $\lambda _{1}<1$ such that for all $f\in B_{s},$ $\mu \in
P_{w},$ $i\in \mathbb{N}$%
\begin{equation}
\begin{array}{c}
||L_{i}\mu||_{''1''}\leq ||\mu||_{''1''} \\ 
||L_{i}\mu||_{s}\leq \lambda _{1}||\mu||_{s}+B||\mu||_{''1''}.
\end{array}
.  \label{1}
\end{equation}
First we remark that obviously
\begin{equation}
||L^{(n)}\mu\Vert _{''1''}\leq ||\mu\Vert _{''1''}.
\end{equation}%
For the stronger norm $||\ ||_s$, given some $j\in \mathbb{N}$, composing the operators we have 
\begin{equation*}
||L_{j}f\Vert _{s}\leq \lambda _{1}\Vert f\Vert _{s}+B\Vert f\Vert _{''1''}
\end{equation*}%
thus%
\begin{eqnarray*}
||L_{j}\circ L_{j+1}(f)\Vert _{s} &\leq &\lambda _{1}\Vert L_{j}f\Vert
_{s}+B\Vert L_{j}f\Vert _{''1''} \\
&\leq &\lambda _{1}^{2}\Vert f\Vert _{s}+\lambda _{1}B||f||_{''1''}+B\Vert
f\Vert _{''1''} \\
&\leq &\lambda _{1}^{2}\Vert f\Vert _{s}+(1+\lambda _{1})B\Vert f\Vert _{''1''}.
\end{eqnarray*}

Continuing the composition, noting that the second coefficient keeps being bounded by a geometric sum we get $(\ref{lyw})$.

\end{proof}

Now we are ready to prove that the iterates of the Lebsgue measure have
bounded Lipschitz multipliers.

\begin{lemma} \label{mult}
Let $\mu _{0}$ be the Lebesgue measure on $X$. The set $$A:=\{L^{(0,k)}\mu
_{0},k\in \mathbb{N}\}$$ has bounded Lipschitz multipliers. There is $C_{A}$
such that for each $i$ and $g\in Lip(\mathbb{S}^{1})$%
\begin{equation*}
||g\mu _{i}||_{s}\leq C_{A}||g||_{Lip}.
\end{equation*}
\end{lemma}

\begin{proof}
We have that $||g\mu ||_{s}=||g\mu ||_{''1''}+||\phi _{g\mu }||_{_{W^{1,1}}}$%
. We first show that for each $\mu \in B_{s}$ the weak norm has bounded
Lipschitz multipliers: 
\begin{equation}
||g\mu ||_{''1''}=\int_{S^{1}}{W(g\mu ^{+}|_{\gamma_l },g\mu
^{-}|_{\gamma_l })}dm(l)\leq 2||g||_{Lip}||\mu ||_{''1''}.  \label{234}
\end{equation}

In order to prove this it is sufficient to show that for each leaf $\gamma $ considering $g\mu
|_{\gamma }$ we have ${W(g\mu ^{+}|_{\gamma },g\mu ^{-}|_{\gamma
})}\leq 2 ||g||_{Lip}{W(\mu ^{+}|_{\gamma },\mu ^{-}|_{\gamma })}$. \
Indeed consider $f~$such that $~Lip(f)\leq 1,~||f||_{\infty }\leq 1$. We
have that also $Lip({f}\frac{{g}}{{||g||}_{Lip}})\leq 2$, $||f\frac{{g}}{{%
||g||}_{Lip}}||_{\infty }\leq 1$ then
\begin{eqnarray*}
{W(g\mu ^{+}|_{\gamma },g\mu ^{-}|_{\gamma })} &{=}%
&\sup_{f~s.t.~Lip(f)\leq 1,~||f||_{\infty }\leq 1}\left\vert \int {f~}d[{%
g\mu ^{-}|_{\gamma }]}-\int {f~}d[{g\mu ^{+}|_{\gamma }}]\right\vert \\
&=&{||g||}_{Lip}\sup_{f~s.t.~Lip(f)\leq 1,~||f||_{\infty }\leq 1}\left\vert
\int {f}\frac{{g|}_{\gamma }}{{||g||}_{Lip}}{~}d[{\mu ^{-}|_{\gamma }]}%
\right. \\
&&\left. -\int {f\frac{{g|}_{\gamma }}{{||g||}_{Lip}}~}d[{\mu ^{+}|_{\gamma }%
}]\right\vert \\
&\leq &2{||g||}_{Lip}{W(g\mu ^{+}|_{\gamma },g\mu ^{-}|_{\gamma }).}
\end{eqnarray*}

From this, integrating we obtain \ref{234}. Since by  Proposition \ref{weakcontral11234}, $%
||L^{(0,k)}\mu _{0}||_{''1''}$ is uniformly bounded as $k\rightarrow \infty $
there is $C_{1,A}$ such that 
\begin{equation*}
||g\mu _{i}||_{''1''}\leq C_{1,A}||g||_{Lip}
\end{equation*}%
for each $i$.

Now we prove that there is $C\geq 0$ such that for all $i,$ 
\begin{equation}
||\phi _{g\mu _{i}}||_{{W^{1,1}}}\leq ||g||_{Lip}C.  \label{ci}
\end{equation}%
Let $l\in \mathbb{S}^1$. We have
\begin{eqnarray*}
|\limsup_{\delta \rightarrow 0} \frac{\phi _{g\mu _{i}}(l+\delta)-\phi _{g\mu _{i}}(l)}{\delta}| &=&|\limsup_{\delta \rightarrow 0}%
\frac{\int_{D}1d[g\mu _{i}|_{\gamma_{l +\delta} }]-\int_{D}1d[g\mu _{i}|_{\gamma_l}]}{%
\delta }| \\
&=&|\limsup_{\delta \rightarrow 0}\frac{\int_{D}g(l +\delta ,y)d\mu
_{i}|_{\gamma_{ l+\delta} }(y)-\int_{D}g(l ,y)~d\mu _{i}|_{\gamma_l }(y)}{%
\delta }| \\
&\leq &|\limsup_{\delta \rightarrow 0}\frac{\int_{D}g(l +\delta ,y)d\mu
_{i}|_{\gamma_{l +\delta} }(y)-\int_{D}g(l+\delta ,y)~d\mu _{i}|_{\gamma_l
}(y)}{\delta }| \\
&&+|\limsup_{\delta \rightarrow 0}\frac{\int_{D}g(l +\delta ,y)d\mu
_{i}|_{\gamma_l }(y)-\int_{D}g(l ,y)~d\mu _{i}|_{\gamma_l }(y)}{\delta }| \\
&\leq &||g||_{Lip}|D(\mu _{i},l )|+||g||_{Lip}||\mu _{i}|_{\gamma_l }||_W.
\end{eqnarray*}
This shows that $\phi_{g\mu_i}$ is absolutely continuous and then in the  Sobolev space $W^{1,1}$ furthermore, integrating over $\mathbb{S}^{1},$ $(\ref{ci})$ \ is satistifed with $%
C=||\mu _{i}||_{{{''{W^{1,1}}''}}}+||\mu _{i}||_{s}$.
Since by Corollary \ref{LYYYYYY} we have that $||\mu _{i}||_{{{''{W^{1,1}}''}}}$ is uniformly bounded as $i$ vary we establish the Lemma.
\end{proof}

\subsection{Superpolynomial convergence to equilibrium for the family of Solenoidal maps and the proof of Proposition \ref{solw}.}

Now we apply the results of the Apendix, Section \ref{sec1} to a family of solenoidal maps satisfying the assumptions $(a),(b),(c)$ stated at the beginning of Section \ref{yp}.

\begin{proposition}\label{wlos}
Let $F_{i}$ be a a sequence of maps satisfying the assumptions $(a),(b),(c)$. Let $B_s$ be the space defined in \eqref{BS}.
Let $\mu\in B_s$ be the invariant probability measure of the limit map $F_0$. Let $L_{F_i}$ the sequence of transfer operators associated to $F_i$.
Then the sequence $L_{F_i}$ ha superpolynomial
convergence to equilibrium to $\mu$ in the following strong sense.  Denoting as before 
$L^{(j,j+n-1)}:=L_{F_{j+n-1}}\circ ... \circ L_{F_{j}}$, there are $C,\lambda\geq 0$  such that $\forall j,n\in \mathbb{N}$, $\mu _{0}\in B_{s}$
\begin{equation}\label{23}
||\mu -L^{(j,j+n-1)}\mu _{0}||_{s}\leq \Phi(n)max(1,||\mu -\mu _{0}||_{s}).
\end{equation}
\end{proposition}

\begin{proof}%[Proof (of Proposition \ref{wlos})]
We will apply Lemma  \ref{losmem} to the family of transfer operators $L_{F_i}$ using as strong space $B_s$ the one defined in \eqref{BS} and as a weak space $B_w$ the one defined in \eqref{l1likespace}.
By Lemma \ref{17} and \eqref{weakcontral11234} the action of the transfer operators on these two spaces satisfy the assumption $(ML1)$. By assumption (c) $(ML2)$ is satisfied. By Proposition \ref{5.8} we have that $(ML3)$ is satisfied.
We can then apply Lemma \ref{losmem} and get that there are $N,M\geq0$ such that for any $j\geq N$   
\begin{equation*}
||L^{(j,j+M-1)}(\mu -\mu _{0})||_{s}\leq \frac{1}{2}
||\mu -\mu _{0}||_{s}.
\end{equation*}
Since
\begin{equation*}
L^{(j,j+M-1)}(\mu -\mu _{0})=L^{(j,j+M-1)}(\mu)-L^{(j,j+M-1)}(\mu_0)
\end{equation*}
by $(c)$, considering that the map only changes on the leaves, where the Wasserstein like distance is considered on positive measures, by $(c)$ we have
\begin{equation*}
||L^{(j,j+M-1)}(\mu)-\mu||_s\leq M\Phi(j)
\end{equation*}
and then
\begin{equation*}
||\mu- L^{(j,j+M-1)}(\mu _{0})||_{s}\leq \frac{1}{2}
||\mu -\mu _{0}||_{s} + M\Phi(j).
\end{equation*}

Denoting   $d_k:=||L^{(j,j+kM-1)}(\mu_0)-\mu||_s $, for $k\geq 0$ the above computation shows that $d_0:=||\mu_0-\mu||_s $,
$d_{k+1}\leq\frac{1}{2}d_k+M \Phi(j+Mk)\leq \frac{1}{2}d_k+M \Phi(j+Mk)[max(1,d_0)]$ showing that $d_k$ decreases superpolynomially fast, satisfying \eqref{23}\footnote{If we have $d_{k+1}\leq\frac{1}{2}d_k+a_n$ with $a_n$ decreasing superpolynomially, then one can rewrite the relation as $d_{k+1}-2a_n\leq\frac{1}{2}(d_k-2a_n)  $ showing that $d_{k+1}$  converges to $2a_n$ exponentially fast.
}.

%\begin{equation*}
%||L^{(j,j+2M-1)}(\mu -\mu _{0})||_{s}\leq \frac{1}{2}
%||L^{(j,j+M-1)}(\mu -\mu _{0})||_{s}.
%\end{equation*}

\end{proof}

\begin{proof}[Proof (of Proposition \ref{solw})]
 The proof of the statement directly follows from the application of Theorem \ref{main}.

The boundedness of the set  $A=\{L^{(0,k)}\mu
_{0},k\in \mathbb{N}\}$ in $B_s$ and of the Lipschitz multipliers is verified in Lemma \ref{mult}, the superpolynomial strong convergence to equilibrium for the family of maps we consider is verified in Lemma \ref{wlos}. The assumption \eqref{preimage} is trivially verified.
This provides the assumptions necessary to apply Theorem \ref{main}, establishing the result.
 
\end{proof}

\begin{remark}
    We remark that in order to get a logarithm law as in Theorem \ref{main} for an eventually autonomous system like the ones considered in this section, a quantitative bound on the speed of convergence of the sequence of the sequence of maps $F_i\to F_0$ is necessary. 
    Let us indeed consider $i\geq 1$ and a family with a slow convergence like $F_i(x,y)=(2x  \mod(1), [\begin{smallmatrix} 
	\frac{1} {\sqrt{i}}  \\ 
	 0
	\end{smallmatrix}]  )$.
for this family of maps the limit map is $F_0$ with    $F_0(x,y)=(2x  \mod(1),  [\begin{smallmatrix} 
	0  \\ 
	 0
	\end{smallmatrix}])$ whose physical measure is the one dimensional lebesgue measure on $S:=\mathbb{S}^1\times \{0\}.$ So $d_\mu(x)=1$ for each $x\in S$. Let us fix $y\in S$ as a target point. Let us consider $i\in \mathbb N $ and an initial condition $x_0$ such that $d(x_0, x)> 0$. We  remark that because of the slow convegence to $F_i$ to $F_0$ we have that for $i,j\geq 1$, $d(F^{(j)}(x_0), x)\leq \frac{1}{i}$ implies that $j\geq i^2$.
This implies that for this system $\liminf_{r\to 0}\frac{log(\tau_r(x_0,y))}{-log(r)}\geq2>d_\mu(y).$
Showing that a logarithm law as in Theorem \ref{main} cannot hold in this case.
\end{remark}

\section{Application to mean field coupled maps}
\label{coupmap}
In this section we show how to apply our main results to a system of mean
field coupled expanding maps, obtaining a logarithm law for this kind of systems. 

To this goal, we will use known results on the convergence to equilibrium of mean field coupled systems.
Those results allow to treat the dynamics of a typical subsystem of the mean field coupled system  as a nonautonomous system having fast convergence to equilibrium (see Definition \ref{def1}). 
 This  general idea is hence applied in this section to a particularly simple class of chaotic systems.

\subsection{A model for infinitely many mean field coupled maps}

We now define a model for the dynamics of an infinite family of expanding maps
interacting in the mean field. 
The mean field system will be composed by infinitely many interacting subsystems, where the dynamics is given by some expanding map, perturbed deterministically by the state of all the other systems in a way which we are going to describe in this subsection.

The phase
space for each interacting subsystem is the unit circle $\mathbb{S}^{1}$, we
will equip $\mathbb{S}^{1}$ with the Borel $\sigma -$algebra.

Let us consider an
additional metric space $M$ \ equipped with the Borel $\sigma -$algebra and
a probability measure $p\in PM(M)$. \ Let us consider a collection of {%
identical $C^{6}$ expanding maps} $(\mathbb{S}^{1},T)_{i}$, with $i\in M$. An admissible global state for the dynamics of this extended system at some
time $t$ is given by a measurable function $\mathbf{x}_{t}:M\rightarrow 
\mathbb{S}^{1}$ associating to every $i\in M$ the state $x_{0}(i)$ of the
subsystem $(\mathbb{S}^{1},T)_{i}$.

We say that the global state $\mathbf{x}_{t}$ of the system is represented
by a probability measure $\mu _{\mathbf{x}_{t}}\in PM(\mathbb{S}^{1})$ if 
\begin{equation*}
\mu _{\mathbf{x}_{t}}=L_{\mathbf{x}_{t}}(p)
\end{equation*}%
(the pushforward of $p$ by the function $\mathbf{x}_{t}$). Let $\mathcal{X}$
be the set of such measurable functions $M\rightarrow \mathbb{S}^{1}$
defining the admissible global states of the system. We now define the dynamics of the
interacting systems by defining a global map $\mathcal{T}:\mathcal{X}%
\rightarrow \mathcal{X}$ and global trajectory of the system by%
\begin{equation*}
\mathbf{x}_{t+1}:=\mathcal{T}(\mathbf{x}_{t})
\end{equation*}%
\ where $\mathbf{x}_{t+1}$ is defined on every coordinate by applying at
each step the common local dynamics $T$, {plus a perturbation given by the
mean field interaction with the other systems}, by%
\begin{equation}\label{mex}
x_{t+1}(i)=\Phi _{\delta ,\mathbf{x}_{t}}\circ T(x_{t}(i))
\end{equation}%
for all $i\in M$, \ where $\Phi _{\delta ,\mathbf{x}_{t}}:\mathbb{S}%
^{1}\rightarrow \mathbb{S}^{1}$ is a diffeomorphism near to the identity when $\delta $ is small  and represents the
perturbation provided by the global mean field coupling. 
 Let us consider  a coupling function $\ h\in C^{6}(\mathbb{
S}^{1}\times \mathbb{S}^{1}\rightarrow \mathbb{R})$.
The function $h(x,y)$ represents the way in which the
presence of some subsystem in the state $y\in \mathbb{S}^{1}$ perturbs a
certain subsystems in the state $x\in \mathbb{S}^{1}$.
The mean field perturbation $\Phi _{\delta ,\mathbf{x}_{t}}$ with strength $
\delta \geq 0$ is defined in the following way: let $\pi _{\mathbb{S}^{1}}: \mathbb{R}\rightarrow \mathbb{S}^{1}$ be the universal covering projection,
; we define \ $\Phi
_{\delta ,\mathbf{x}_{t}}$ as 
\begin{equation}
\Phi _{\delta ,\mathbf{x}_{t}}(x):=x+\pi _{\mathbb{S}^{1}}(\delta \int_{%
\mathbb{S}^{1}}h(x,y)~d\mu _{\mathbf{x}_{t}}(y)).  \label{diffeo}
\end{equation}
We remark that in this definition the parameter $\delta$ plays the role of the strength of the coupling.
Since \eqref{mex} is clearly a measurable map  we see that the measure representing the current state of the system fully
determines the measure which represents the next state of the system,
defining a function between measures $\mu _{\mathbf{x}_{t}}\rightarrow \mu _{%
\mathbf{x}_{t+1}}$ defined as%
\begin{equation*}
\mu _{\mathbf{x}_{t+1}}=L_{\Phi _{\delta ,\mu _{_{\mathbf{x}_{t}}}}\circ
T}(\mu _{\mathbf{x}_{t}}):=\mathcal{L}_{\delta }(\mu _{\mathbf{x}_{t}}).
\end{equation*}
Now, let us consider $\delta \geq 0$ and denote by $(\mathbb{S}^{1},T,\delta ,h)$ the extended
system in which these maps are coupled by  $h$ as explained above.
The function $\mathcal{L}_{\delta }$ is also called to be the Self Consistent Transfer Operator associated to the mean field coupled system $(\mathbb{S}^{1},T,\delta ,h)$.

Since in every subsystem and coordinate, at each iteration,
the map $\Phi _{\delta ,\mu _{_{\mathbf{x}_{t}}}}\circ T$ is applied, if we observe the evolution of a single coordinate,
 we see the result of the application of a nonautonomous dynamical system $(\mathbb{S}%
^{1},T_{n})$ where $T_{n}=\Phi _{\delta ,\mu _{_{\mathbf{x}_{n}}}}\circ T$.
\  

The transfer operators associated to expanding maps are well known to preserve absolutely continuous measures (see \cite{Gdisp}) and in particular measures having a density \footnote{We say that a measure $\mu$ on $\mathbb{S}^1$ has a density $f_\mu$ if $f_\mu=\frac{d\mu}{dm}$, the Radon Nikodym derivative of $\mu$ with respect to the Lebesgue measure $m$ on $\mathbb{S}^{1}$.} in the Sobolev space $W^{1,1}$.
For this reason we will consider such space as a strong space $B_s $ in the following.
In the case where $T$ is an expanding map and $x_{0}$ is represented by a measure $\mu_{x_0}$ which is smooth enough we can establish a logarithm law for the dynamics of each coordinate.

For this kind of extended system we prove:

\begin{proposition}\label{explog}
Let us fix $i\in M$ and let $x_t(i)$ the evolution of the $i-$th coordinate of the mean field coupled  system  $(\mathbb{S}^{1},T,\delta ,h)$ as defined above. Let us suppose that  the global initial condition of the system is distributed in a smooth way, that is $\mu_{x_0}$ is an absolutely continuous measure having density in $W^{1,1}$. \footnote{We remark that the global initial distribution $\mu_{x_0}$ and its evolution in time does not depend on the single $i-$th subsystem initial condition $x_0(i)$.} 
We will also suppose that the coupling is small. In the sense that there is $\hat \delta>0$ such that for each $0\leq\delta \leq \hat \delta$ the following result will hold.
We define the hitting time of a small target centered at $y$ for the $i-$th subsystem with initial condition  $x_0(i)$ as $$\tau_r(x_0(i),y)=sup(\{t\geq 0| d(x_t(i),y) \geq r\}).$$
Let $m$ be the Lebesgue measure on $\mathbb{S}^{1}$.
Then for each $y\in \mathbb{S}^{1}$ and $m$ almost each  $x_0(i)$  it holds
\begin{equation}\label{112}
\underset{r\rightarrow 0}{\lim }\frac{\log \tau _{r}(x_0(i),x)}{-\log r}%
=1.
\end{equation}
   
\end{proposition}

To prove Proposition \ref{explog} we need some preliminary results  we will take from \cite{GalCMP}, Section 7.
The following statement shows that our mean field coupled system has a unique regular invariant measure when \ $\delta $ is small enough.

\begin{proposition}[Existence and uniqueness of the invariant measure]
\label{existenceexp}Let  $(\mathbb{S}^{1},T,\delta ,h)$ as above and let and $\mathcal{L}_{\delta }$ be the associated self consistent transfer operator. If $\delta >0$ is small enough then there is a unique probability
measure $\mu _{\delta }$\ having density $f_{\delta }\in L^{1}$ such that%
\begin{equation*}
\mathcal{L}_{\delta }(\mu _{\delta })=\mu _{\delta }.
\end{equation*}

Furthermore $f_{\delta }\in W^{5,1}.$
\end{proposition}

The following statement is an estimate for the speed of convergence to
equilibrium of mean field coupled expanding maps (see \cite{Kel}, or \cite{TaSe}
for similar statements).

\begin{proposition}[Exponential convergence to equilibrium]
\label{convmaps}Let $\mathcal{L}_{\delta }$ be the family of self-consistent
transfer operators arising $(\mathbb{S}^{1},T,\delta ,h)$  as above.
Let $\mu_{\delta }$ be the absolutely continuous invariant
probability measure of $\mathcal{L}_{\delta }.$ Let us denote by $f_\delta\in W^{1,1}$  the density of $\mu_\delta$ with respect to the Lebesgue measure $m$. Then there exists $\overline{%
\delta }>0$ \ and $C,\gamma \geq 0$ such that for all $0<\delta <\overline{%
\delta }$, and each probability measure\ $\nu $ having density $f_{\nu }\in
W^{1,1}$ we have%
\begin{equation*}
||\frac{d}{dm}\mathcal{L}_{\delta }^{n}(\nu )-f_{\delta }||_{W^{1,1}}\leq Ce^{-\gamma
n}||f_\nu -f_{\mu _{\delta }}||_{W^{1,1}}
\end{equation*}
were we recall that the notation $\frac{d}{dm}$ represents the Radon Nykodim derivative with respect to the Lebesgue meaure $m$.
\end{proposition}

We can now apply Theorem \ref{main} \ to get \
a logarithm law result for mean field coupled maps.

%\begin{theorem}
%Let a system of mean field coupled maps satisfying the assumptions $1),2),3)$
%above, suppose that the global initial condition $\mathbf{x}_{0}$ is
%reporesented by a measure $\mu _{0}$ having a density in %$W^{1,1}$, then for
%each $y\in \mathbb{S}^{1}$ $i\in M$ we have that%
%\begin{equation}
%\underset{r\rightarrow 0}{\lim }\frac{\log \tau _{r}(x,x_{0})}{-\log r}=1
%\label{112}
%\end{equation}%
%holds for Lebesgue almost all %5$x\in \mathbb{S}^{1}.$
%\end{theorem}

\begin{proof}[Proof of Proposition \ref{explog}]
We will get the result by a direct application of Theorem \ref{main} to the nonautonomous system
$(\mathbb{S}^{1},T_{i})$ where $T_i=\Phi_{\delta,\mu_{x_i}}\circ T$ and  $\mu_{x_i}$ is the measure representing the global state at time $i$, satisfying $\mu_{x_i}=\mathcal{L}^i(\mu_{x_o})$.
We will consider as a strong space $B_s$ the space of signed measures having a density in  $W^{1,1}$ with the topology induced by the one  on $W^{1,1}$. We remark that this topology is stronger than then one induced by the $W$ distance. Furthermore we remark that by Proposition \ref{convmaps} the set $$A=\{\mathcal{L}^i(\mu_{x_o})\}_{i\in \mathbb{N}}$$ is bounded in $B_s$ and has obviously bounded Lipshitz multipliers.
By Proposition \ref{convmaps} we also see that $\mu _{n}:=\mathcal{L}^i(\mu_{x_o})$ converge exponentially fast to the invariant measure 
$\mu_\delta \in W^{1,1}$, then Theorem \ref{main} can be
applied. We remark that since $\mu_\delta$ has density $f_\delta \in W^{1,1}$, $d_{\mu }(y)=1$ $\forall
y\in \mathbb{S}^{1}$ establishing \ref{112}.
\end{proof}

\section{ Appendix: exponential loss of memory for sequential composition of operators \label{sec1}}

In this section, we show a relatively simple and general argument that establishes exponential loss of memory for a sequential composition of Markov operators converging to a limit. This is used as a tool in Section \ref{yp} to establish a fast convergence to equilibrium for a sequential composition of solenoidal maps.
 The results we present are similar to those of \cite{CR07}, although proved in a more general context, allowing the application to solenoidal maps.
The methods used are also inspired by the constructive methods used in \cite{GNS15}.
Since the approach is general, we will work in an abstract framework, stating a result that holds for a sequence of Markov operators acting on suitable spaces of measures.
Let $B_{w}$  and $B_{s}$ be  normed vector subspaces of signed measures on $X$.
Suppose $(B_{s},||~||_{s})\subseteq $ $(B_{w},||~||_{w})$ and $||~||_{s}\geq
||~||_{w}$.
Let us consider a sequence of Markov operators $\{L_{i}\}_{i\in \mathbb{N}%
}:B_{s}\rightarrow B_{s}.$ \ We will suppose furthermore that the following 
assumptions are satisfied by the $L_{i}$:

\begin{itemize}
\item[$ML1$] The operators $L_{i }$ satisfy a common "one step"
Lasota Yorke inequality. There are constants $B,\lambda
_{1}\geq 0$ with $\lambda _{1}<1$ such that for all $f\in B_{s},$ $\mu \in
P_{w},$ $i\in \mathbb{N}$%
\begin{equation}
\left\{ 
\begin{array}{c}
||L_{i}f||_{w}\leq ||f||_{w} \\ 
||L_{i}f||_{s}\leq \lambda _{1}||f||_{s}+B||f||_{w}.%
\end{array}%
\right.  \label{1}
\end{equation}

\item[$ML2$] There is a Markov operator $L_0:B_s\to B_s$ having an invariant probability measure $\mu \in B_s$ such that the family of operators satisfy: $\lim_{i\rightarrow +\infty
}L_{i}=L_{0}$ in the $B_{s}\rightarrow B_{w}$ topology.
\footnote{In particular the family of operators satisfy: $\forall \epsilon >0~\exists
N~s.t.~\forall i,j\geq N$ 
\begin{equation}
||(L_{i}-L_{j})||_{B_{s}\rightarrow B_{w}}\leq \epsilon .  \label{nn}
\end{equation}
}

\item[$ML3$] There exists $a_{n}\geq 0$ with $a_{n}\rightarrow 0$ such that
for all $n\in \mathbb{N}$ and $v\in V_{s}$%
\begin{equation}
||L_{0}^{n}(v)||_{w}\leq a_{n}||v||_{s}  \label{3}
\end{equation}%
where%
\begin{equation*}
V_{s}=\{\mu \in B_{s}|\mu (X)=0\}.
\end{equation*}
\end{itemize}

We recall that since $\mu \rightarrow \mu (X)$ is continuous, $V_{s}$ is
closed. Furthermore $\forall i,$ $L_{i
}(V_{s})\subseteq V_{s}$.

We remark that the assumption $(ML1)$ implies that the family of operators $%
L_{i}$ is uniformly bounded when acting on $B_{s}$ and on $B_{w}.$

First, we establish a Lasota-Yorke inequality for a sequential composition of operators satisfying $(ML1)$. The proof of the Lemma is essentially the same as the proof of Corollary \ref{trima}.

\begin{lemma}
\label{lasotaY copy(1)}Let $L_{i}$ be a family of Markov operators \
satisfying $(ML1)$ and let
\begin{equation}
L^{(j,j+n-1)}:=L_{j}\circ L_{j+1}\circ ...\circ ~L_{j+n-1}  \label{Ln}
\end{equation}%
be a sequential composition of operators in such family, then $\forall n,j$%
\begin{equation}
||L^{(j,j+n-1)}f\Vert _{w}\leq ||f\Vert _{w}
\end{equation}%
and%
\begin{equation}
||L^{(j,j+n-1)}f\Vert _{s}\leq \lambda _{1}^{n}\Vert f\Vert _{s}+\frac{B}{1-\lambda
_{1}}\Vert f\Vert _{w}.  \label{lyw}
\end{equation}
\end{lemma}

%\begin{proof}
%The first inequality is straightforward from $(ML1)$. Let us now prove $(%
%\ref{lyw})$. We have 
%\begin{equation*}
%||L_{j}f\Vert _{s}\leq \lambda _{1}\Vert f\Vert _{s}+B\Vert f\Vert _{w}
%\end{equation*}%
%thus%
%\begin{eqnarray*}
%||L_{j}\circ L_{j+1}(f)\Vert _{s} &\leq &\lambda _{1}\Vert L_{j}f\Vert
%_{s}+B\Vert L_{j}f\Vert _{w} \\
%&\leq &\lambda _{1}^{2}\Vert f\Vert _{s}+\lambda _{1}B||f||_{w}+B\Vert
%f\Vert _{w} \\
%&\leq &\lambda _{1}^{2}\Vert f\Vert _{s}+(1+\lambda _{1})B\Vert f\Vert _{w}
%\end{eqnarray*}
%
%Continuing the composition we get $(\ref{lyw})$.
%\end{proof}

The following lemma is an estimate for the distance of the sequential composition of operators from the iterations of $L_0$.

\begin{lemma}
\label{XXX}Let $\delta \geq 0$ and let $ L^{(j,j+n-1)}$ be a sequential composition
of operators $\{L_{i}\}_{i\in \mathbb{N}}$ as in $(\ref{Ln})$ that satisfies the above assumptions.
Let $L_0$ as above such that $||L_{i}-L_0||_{s\rightarrow w}\leq \delta .$\ Then
there is $C\geq 0$ such that $\forall g\in B_{s},\forall j,n\geq 1$%
\begin{equation}
||L^{(j,j+n-1)}g-L_0^{n}g||_{w}\leq \delta (C||g||_{s}+n\frac{B}{1-\lambda }||g||_{w}).
\label{2}
\end{equation}%
where $B$ is the second coefficient of the Lasota Yorke inequality $($\ref{1}%
$)$.
\end{lemma}

\begin{proof}
\ By the assumptions we get%
\begin{equation*}
||L_0g-L_{j}g||_{w}\leq \delta ||g||_{s}
\end{equation*}

hence the case $n=1$ of $(\ref{2})$ is trivial. \ Let us now suppose inductively%
\begin{equation*}
||L^{(j,j+n-2)}g-L_{0}^{n-1}g||_{w}\leq \delta (C_{n-1}||g||_{s}+(n-1)\frac{B}{%
1-\lambda _{1}}||g||_{w})
\end{equation*}%
then%
\begin{eqnarray*}
||L_{j+n-1}L^{(j,j+n-2)}g-L_{0}^{n}g||_{w} &\leq
&||L_{j+n-1}L^{(j,j+n-2)}g-L_{j+n-1}L_{0}^{n-1}g+L_{j+n-1}L_{0}^{n-1}g-L_{0}^{n}g||_{w} \\
&\leq
&||L_{j+n-1}L^{(j,j+n-2)}g-L_{j+n-1}L_{0}^{n-1}g||_{w}+||L_{j+n-1}L_{0}^{n-1}g-L_{0}^{n}g||_{w}
\\
&\leq &\delta (C_{n-1}||g||_{s}+(n-1)\frac{B}{1-\lambda _{1}}%
||g||_{w})+||[L_{j+n-1}-L_{0}](L_{0}^{n-1}g)||_{w} \\
&\leq &\delta (C_{n-1}||g||_{s}+(n-1)\frac{B}{1-\lambda _{1}}%
||g||_{w})+\delta ||L_{0}^{n-1}g||_{s} \\
&\leq &\delta (C_{n-1}||g||_{s}+(n-1)\frac{B}{1-\lambda _{1}}||g||_{w}) \\
&&+\delta (\lambda _{1}^{n-1}||g||_{s}+\frac{B}{1-\lambda _{1}}||g||_{w}) \\
&\leq &\delta \lbrack (C_{n-1}+\lambda _{1}^{n-1})||g||_{s})+n\frac{B}{%
1-\lambda _{1}}||g||_{w}].
\end{eqnarray*}

The statement follows from the observation that, continuing the composition, $ C_{n}$ remains bounded by the sum of a geometric series.
\end{proof}

\bigskip

\bigskip

\begin{lemma}\label{losmem}
Let $L_{i}$ be a sequence of operators satisfying $(ML1),...,(ML3)$.
%as in Lemma \ref{XXX} and \ref{lasotaY copy(1)}

Then the sequence $L_{i}$ has a strong exponential loss of memory in the following sense.  There are $C,\lambda\geq 0$  such that $\forall j,n\in \mathbb{N}$, $g\in V_{s}$
\begin{equation*}
||L^{(j,j+n-1)}g||_{s}\leq Ce^{-\lambda n}||g||_{s}.
\end{equation*}
\end{lemma}

\begin{proof}
It is standard to deduce from the assumptions that $\mu $ is the unique invariant probability measure of $L_0$ in $B_s$. Now, consider $\mu _{0}\in B_{s}$. 
 %denote for each $i\geq 1$, $\mu _{i}:=L^{(j,j+i-1)}\mu _{0},$.
 Remark that because of the Lasota-Yorke inequality, $\forall j,i\geq 1,g\in B_s $ $$||L^{(j,j+i)}(g)||_s\leq  (\frac{B}{1-\lambda_1}+1)||g||_s.$$
 Now let us onsider $N_0$ such that $\lambda_1^{N_0}\leq \frac{1}{100(\frac{B}{1-\lambda_1}+1)}$ and by $(ML3)$, $N_{2}$ such
that $\forall i\geq N_{2},g\in V_s$ $$||{L_0}^{N_2}g||_{w}\leq \frac{1}{100B}||g||_s.$$ 
Let $M:=\max
(N_0,N_{2}).$
Let $N_1$ such that $$
||L_{i}-L_0||_{s\rightarrow w}\leq \frac{(1-\lambda_1 )}{100M B({C+B})}$$ for all $i\geq N_{1}$. By $(\ref{2})$, $\forall j\geq N_1, i\geq M$
\begin{equation*}
\begin{split}
||L^{(j,j+i-1)}g-{L_0 }^{i}g||_{w} & \leq \frac{(1-\lambda_1 )}{100MB(C+B)}(C||g||_s+i\frac{B||g||_w}{(1-\lambda_1 )})\\
& \leq \frac{i}{100MB}||g ||_s.
\end{split}
\end{equation*}%
 Hence 
 \begin{equation}
 \begin{split}
     ||L^{(j,j+i-1)}g||_w &\leq ||{L_0 }^{i}g||_w+ \frac{i}{100MB}||g ||_s\\
     &\leq 
     \frac{1}{100B}||g||_s + \frac{i}{100MB}||g ||_s.
\end{split}    
\end{equation}

Applying now the Lasota-Yorke inequality we get, for any $j\geq N_1$
\begin{equation}
 \begin{split}
||L^{(j,j+2M-1)}g\Vert _{s} & \leq \lambda
_{1}^{-M}||L^{(j,j+M-1)}g||_s+B\Vert L^{(j,j+M-1)}g\Vert _{w} \\
&\leq  \frac{1}{100}||g||_s+B \frac{1}{100B}||g||_s+\frac{BM}{100MB}||g ||_s \\
&\leq\frac{3}{100}||g||_s
\end{split}    
\end{equation}
and $$||L^{(j,j+2kM-1)}g\Vert _{s}\leq\frac{3}{100}^k||g||_s$$ for each $j\geq N_1$ and $k\geq 1$, $g\in V_s$ establishing the  result.
\end{proof}

\bigskip

\noindent {\bf Acknowledgements:}
S.G. acknowledges the MIUR Excellence Department Project awarded to the
Department of Mathematics, University of Pisa, CUP I57G22000700001.
S.G. was partially supported by  the
research project "Stochastic properties of dynamical systems" (PRIN 2022NTKXCX) of the Italian Ministry of Education and Research.

\noindent {\bf Declaration:} The authors have no competing interests to declare that are relevant to the content of this article.

\bibliographystyle{ieeetr}
\bibliography{extremes.bib}% Produces the bibliography via BibTeX.

% \begin{thebibliography}{BGI}
% \bibitem[BS]{BS} Barreira L, Saussol B,\ \emph{Hausdorff dimension of
% measures via Poincar\'{e} recurrence}, Commun. Math. Phys., \textbf{219}
% (2001), 443-463.

% \bibitem[BGI]{BGI} Bonanno C, Isola S, Galatolo S \emph{Recurrence and
% algorithmic information} Nonlinearity \textbf{17} (2004), no. 3, 1057--1074.

% \bibitem[Bo]{Bo} Boshernitzan M D, \emph{Quantitative recurrence results},
% Invent. Math. \textbf{113} (1993), 617-631

% \bibitem[CG]{CG} Carletti T. , Galatolo S., \emph{Numerical Estimates of
% dimension by recurrence and waiting time} Physica A\emph{\ (2006).}81--587.

% \bibitem[G]{G} Galatolo, S\emph{.} \emph{Dimension via waiting time and
% recurrence}, Math. Res. Lett. \textbf{12}, no 3, May 2005, 377-386

% \bibitem[GK]{GK} Galatolo S, Kim D H, \emph{The dynamical Borel-Cantelli
% lemma and the waiting time problems.} Preprint, Arxiv: math.DS/0610213.
% \end{thebibliography}

\end{document}